\documentclass[12pt]{article}

\usepackage{amssymb}

\usepackage{enumitem}

\usepackage{graphicx}

\usepackage{marvosym}
\usepackage{mathrsfs}
\usepackage{amsmath}
\usepackage{amsfonts}
\usepackage{enumerate}
\usepackage{verbatim}
\usepackage{yhmath}
\usepackage{amsthm}
\usepackage{color,xcolor}

\allowdisplaybreaks[4]

\makeatletter
\@namedef{subjclassname@2020}{%
  \textup{2020} Mathematics Subject Classification}
\makeatother

\newtheorem{theorem}{Theorem}[section]
\newtheorem{corollary}[theorem]{Corollary}
\newtheorem{lemma}[theorem]{Lemma}

\newtheorem*{theorema}{Theorem A}
\newtheorem*{theoremb}{Theorem B}
\newtheorem*{theoremc}{Theorem C}

\theoremstyle{definition}

\numberwithin{equation}{section}

\begin{document}

\title{Fourier integral operators on Hardy spaces with amplitudes in forbidden H\"{o}rmander class\thanks{ This work was supported by the National Key Research and Development Program of China, grant number 2022YFA1005700
 and the National Natural Science Foundation of China, grant number 11871436.}}
\author{Xiaofeng Ye\hspace{1cm} Chunjie Zhang\hspace{1cm} Xiangrong Zhu\thanks{is the corresponding author.}}

\date{}

\maketitle

 \begin{abstract}
 In this note, we consider a Fourier integral operator defined by
 \begin{align*}
 T_{\phi,a}f(x)=\int_{\mathbb{R}^{n}}e^{i\phi(x,\xi)}a(x,\xi)\widehat{f}(\xi)d\xi,
 \end{align*}
 where $a$ is the amplitude, and $\phi$ is the phase.

 Let $0\leq\rho\leq 1,n\geq 2$ or $0\leq\rho<1,n=1$ and
 $$m_p=\frac{\rho-n}{p}+(n-1)\min\{\frac 12,\rho\}.$$
 If $a$ belongs to the forbidden H\"{o}rmander class $S^{m_p}_{\rho,1}$ and $\phi\in \Phi^{2}$ satisfies the strong non-degeneracy condition,
 then for any $\frac {n}{n+1}<p\leq 1$, we can show that the Fourier integral operator $T_{\phi,a}$ is bounded from the local Hardy space $h^p$ to $L^p$.
 Furthermore, if $a$ has compact support in variable $x$, then we can extend this result to $0<p\leq 1$. As $S^{m_p}_{\rho,\delta}\subset S^{m_p}_{\rho,1}$ for any $0\leq \delta\leq 1$,
 our result supplements and improves upon recent theorems proved by Staubach and his collaborators for $a\in S^{m}_{\rho,\delta}$ when $\delta$ is close to 1.

  As an important special case, when $n\geq 2$, we show that $T_{\phi,a}$ is bounded from $H^1$ to $L^1$ if $a\in S^{(1-n)/2}_{1,1}$ which is a generalization of the
  well-known Seeger-Sogge-Stein theorem for $a\in S^{(1-n)/2}_{1,0}$. This result is false when $n=1$ and $a\in S^{0}_{1,1}$.
 \end{abstract}

\section{Introduction and main results}

 \hspace{4mm}  A pseudo-differential operator (PDO for short) is given by
 $$T_a f(x)=(2\pi)^{-n}\int_{\mathbb{R}^n}e^{ix\cdot\xi}a(x,\xi)\widehat{f}(\xi)d\xi,$$
 where $\widehat{f}$ is the Fourier transform of $f$ and $a$ is the amplitude. We always omit the constant $(2\pi)^{-n}$ throughout this note.

 In its basic form, a Fourier integral operator (FIO for short) is defined by
 \begin{align*}
 T_{\phi,a}f(x)=\int_{\mathbb{R}^{n}}e^{i\phi(x,\xi)}a(x,\xi)\widehat{f}(\xi)d\xi,
 \end{align*}
 where $\phi$ is the phase. In this note, we always assume that $f$ belongs to the Schwartz class $S(\mathbb{R}^{n})$.

 A FIO $T_{\phi,a}$ is simply a pseudo-differential operator if $\phi(x,\xi)=x\cdot \xi$. When $\phi(x,\xi)=x\cdot \xi+|\xi|$, $T_{\phi,a}$ is closely related to the wave equation and Fourier transform on the unit sphere in $\mathbb{R}^{n}$ (see \cite[p. 395]{S93}).

 FIOs have been widely used in the theory of partial differential equations and micro-local analysis. For instance, the solution to an initial value problem for a hyperbolic equation with variable coefficients can be effectively approximated using an FIO of the initial value (see \cite[p. 425]{S93}). Therefore, the boundedness of related FIOs provides a priori estimate for the solution. A systematic study of
 these operators was initiated by H\"{o}rmander \cite{H71AM}.

 At first, we recall some simplest and most useful definitions on amplitudes and phases.

 Let $\mathbb{N}$ be the set $\{0,1,2,\ldots\}$. A function $a$ belongs to the H\"{o}rmander class $S^{m}_{\rho,\delta}$ $(m\in \mathbb{R},0\leq\rho,\delta\leq1)$ if it satisfies
 \begin{equation}\label{fiohp1.1}
 \sup_{x,\xi\in\mathbb{R}^{n}}(1+|\xi|)^{-m+\rho N-\delta M}|\nabla^{N}_{\xi}\nabla^{M}_{x}a(x,\xi)|=A_{N,M}<+\infty
 \end{equation}
 for any $N,M\in \mathbb{N}$. Immediately, one have
 $$S^{m_1}_{\rho,\delta}\subset S^{m_2}_{\rho,\delta},S^{m}_{\rho_2,\delta}\subset S^{m}_{\rho_1,\delta},S^{m}_{\rho,\delta_1}\subset S^{m}_{\rho,\delta_2},$$
 if $m_1<m_2,\rho_1<\rho_2,\delta_1<\delta_2.$

 A real-valued function $\phi$ belongs to the class $\Phi^{2}$ if $\phi$ is positively homogeneous of order 1 in the frequency variable $\xi$ and satisfies
 \begin{equation}\label{fiohp1.2}
 \sup_{(x,\xi)\in\mathbb{R}^{n}\times(\mathbb{R}^{n}\setminus\{0\})}|\xi|^{-1+N}|\nabla^{N}_{\xi}\nabla^{M}_{x}\phi(x,\xi)|=B_{N,M}<+\infty
 \end{equation}
 for all $N,M\in \mathbb{N}$ with $N+M\geq 2$.

 A real-valued function $\phi\in C^{2}(\mathbb{R}^{n}\times(\mathbb{R}^{n}\setminus\{0\}))$ satisfies the strong non-degeneracy condition (SND for short), if there exists a constant $\lambda>0$ such that
 \begin{equation}\label{fiohp1.3}
 \textrm{det}\left(\frac{\partial^{2}\phi}{\partial x_{j}\partial \xi_{k}}(x,\xi)\right)\geq \lambda
 \end{equation}
 for all $(x,\xi)\in\mathbb{R}^{n}\times(\mathbb{R}^{n}\setminus\{0\})$.

 For PDOs and FIOs, the most important problem is whether they are bounded on Lebesgue spaces and Hardy spaces.
 This problem has been extensively studied and there are numerous results. Here we always assume that $a\in S^m_{\rho,\delta}$ and $\phi\in \Phi^{2}$ satisfies the SND condition (\ref{fiohp1.3}).

 For PDOs, if $a\in S^{m}_{\rho,\delta}$ with $\delta<1$ and $m\leq \frac n2\min\{0, \rho-\delta\}$, then $T_a$ is bounded on $L^2$ and the bound on $m$ is sharp.
 See H\"{o}rmander \cite{H71CPAM}, Hounie \cite{H86} and Calder\'{o}n-Vaillancourt \cite{CV71,CV72}. For $a\in S^{m}_{\rho,1}$ Rodino \cite{R76} proved that $T_a$ is bounded on $L^2$ if $m<\frac{n(\rho-1)}{2}$ and
 constructed a amplitude $a\in S^{n(\rho-1)/2}_{\rho,1}$ such that $T_a$ is unbounded on $L^2$.  For endpoint estimates, one can see \'{A}lvarez-Hounie \cite{AH90} and Guo-Zhu \cite{GZ22}.

 For the local $L^{2}$ boundedness of FIO, it can be date back to Eskin \cite{E70} and H{\"o}rmander \cite{H71AM}. The transference of local to global regularity of FIOs can be found in
 Ruzhansky-Sugimoto \cite{RS19}. Among numerous results on the global $L^{2}$ boundedness, we would like to mention that Dos Santos Ferreira-Staubach \cite{FS14} proved the global $L^{2}$ boundedness if either $m\leq\frac{n}{2}\min\{0,\rho-\delta\}$ when $\delta<1$ or $m<\frac{n}{2}(\rho-1)$ when $\delta=1$. This bound on $m$ is also sharp. For more results, see for instance \cite{Af78,B97,F78,K76,RS06}.

 For endpoint estimates of FIOs, Seeger-Sogge-Stein \cite{SSS91} proved the local $H^{1}$-$L^{1}$ boundedness for $a\in S^{(1-n)/2}_{1,0}$ and got the $L^{p}$ boundedness for
 $a\in S^{m}_{1,0}$ when $m=(1-n)|\frac{1}{p}-\frac{1}{2}|$, by the Fefferman-Stein interpolation. Tao \cite{T04} showed the weak type (1,1) for $a\in S^{(1-n)/2}_{1,0}$.
 For the regularity of FIOs and its applications, there has been a great deal of progress and work recently, for example Cordero-Nicola-Rodino \cite{CNR09,CNR10}, Coriasco-Ruzhansky \cite{CR10, CR14},
 Hassell-Portal-Rozendaal \cite{HPR20}, Israelsson-Rodr\'{i}guez L\'{o}pez-Staubach \cite{IRS21} etc.. Among these results, the latest result is the following theorem which is proved in Castro-Israelsson-Staubach \cite{CIS21}.
 \begin{theorema}
 (\cite[Theorem 1]{CIS21}) Let $n\geq 1,0\leq \rho\leq 1, 0\leq \delta<1$,
 $$m=(\rho-n)|\frac 12-\frac 1p|+\frac{n}{2}\min\{0,\rho-\delta\}$$
 and $a\in S^{m}_{\rho,\delta}$. If $\phi\in \Phi^{2}$ satisfies the SND condition (\ref{fiohp1.3}), then the FIO $T_{\phi,a}$ is bounded on $L^p$ for $1<p<\infty$.
 \end{theorema}

 In \cite{SZ23}, when $2<p<\infty$ and $0<\rho\leq 1$, by using a new endpoint estimate, Shen-Zhu improved Theorem A to the following theorem which may be optimal.
 \begin{theoremb}\cite[Theorem 1.2]{SZ23}
 Let $n\geq 1,0<\rho\leq 1, 0\leq \delta<1$,
 $$m=(\rho-n)|\frac 12-\frac 1p|+\frac{n}{p}\min\{0,\rho-\delta\}$$
 and $a\in S^{m}_{\rho,\delta}$. If $\phi\in \Phi^{2}$ satisfies the SND condition (\ref{fiohp1.3}), then $T_{\phi,a}$ is bounded on $L^p$ for $2<p<\infty$.
 \end{theoremb}
 Theorem B remain true when $\rho=0$ because in this case $T_{\phi,a}$ can be considered as the sum of the low frequency part and a PDO. This fact has been also used in Section 3.1.

 Before the next theorem, we recall the definitions of Hardy spaces and local Hardy spaces. Let $\Phi$ be a function in the Schwartz space $S(\mathbb{R}^n)$ satisfying $\int_{\mathbb{R}^n}\Phi(x)dx=1$.
 Set $\Phi_t(x)=\frac{1}{t^n}\Phi(\frac{x}{t})$. Following Stein \cite[p. 91]{S93}, we can define the Hardy space $H^p(\mathbb{R}^n)(0<p<+\infty)$ as the space of all tempered distributions $f$ satisfying
 $$\|f\|_{H^p}=\|\sup\limits_{t>0}|f\ast\Phi_t(x)|\|_{L^p(\mathbb{R}^n)}<\infty.$$
 The local Hardy space (see \cite {G79}) $h^p(\mathbb{R}^n)(0<p<+\infty)$ is defined as the space of all tempered distributions $f$ satisfying
 $$\|f\|_{h^p}=\|\sup\limits_{0<t<1}|f\ast\Phi_t(x)|\|_{L^p(\mathbb{R}^n)}<\infty.$$
 It is well known that $H^p=h^p=L^{p}$ for equivalent norms when $1<p<+\infty$ and $H^p\subset h^p\subset L^p$ when $0<p\leq 1$.

 In \cite{IMS23}, Israelsson-Mattsson-Staubach extend Theorem A to $p\leq 1$ implicity.
 \begin{theoremc}
 (\cite{IMS23}) Let $n\geq 1,0\leq \rho\leq 1, 0\leq \delta<1$,
 $$m=(\rho-n)|\frac 12-\frac 1p|+\frac{n}{2}\min\{0,\rho-\delta\}$$
 and $a\in S^{m}_{\rho,\delta}$. If $\phi\in \Phi^{2}$ satisfies the SND condition (\ref{fiohp1.3}), then $T_{\phi,a}$ is bounded from the local Hardy space $h^p$ to $L^p$ when  $\frac {n}{n+1}<p\leq 1$.
 Furthermore, if $a$ has compact support in the spatial variable $x$, then $T_{\phi,a}$ is bounded from $h^p$ to $L^p$ for any $0<p\leq 1$.
 \end{theoremc}
 Theorem C can be considered as a corollary of \cite[Proposition 5.7]{IRS21}, \cite[Proposition 6.4]{IRS21} and \cite[Proposition 5.1]{IMS23} due to the facts $h^p=F^0_{p,2}$ and $h^p\subset L^p$.

 Set $\lambda_{\rho}=1$ when $0\leq \rho\leq \frac 12$ and $\lambda_{\rho}=\frac{1}{2\rho}$ when $\frac 12<\rho\leq 1$, $\rho_0=\min\{\rho,\frac 12\}$ and
 \begin{equation}\label{fiohp1.4}
 m_{p}=\frac{\rho-n}{p}+(n-1)\min\{\rho,\frac 12\}=\frac{\rho-n}{p}+(n-1)\rho_0.
 \end{equation}
 It is easy to see that
 $$\frac 12\leq \lambda_{\rho}\leq 1,\rho_0=\lambda_{\rho}\rho.$$

 Now we state our main result in this note.
 \begin{theorem}
  Suppose that $0\leq\rho\leq 1,n\geq 2$ or $0\leq\rho<1,n=1$, $a\in S^{m_p}_{\rho,1}$ where $m_p$ is given by (\ref{fiohp1.4}) and $\phi\in \Phi^{2}$ satisfies the SND condition (\ref{fiohp1.3}). Then for any
  $\frac {n}{n+1}<p\leq 1$, we have
 $$\|T_{\phi,a}f\|_{L^p}\leq C\|f\|_{h^p}.$$
 Furthermore, if $a$ has compact support in variable $x$, then for any $0<p\leq 1$ we have
  $$\|T_{\phi,a}f\|_{L^p}\leq C\|f\|_{h^p}.$$
 The constants here depend only on $n,\rho,\lambda,p$ and finitely many semi-norms $A_{N,M}, B_{N,M}$.
 \end{theorem}

 As the most important special case, we generalize the well-known Seeger-Sogge-Stein theorem to the forbidden H\"{o}rmander class $S^{(1-n)/2}_{1,1}$ when $n\geq 2$.
 \begin{corollary}
 If $n\geq 2$, $a\in S^{(1-n)/2}_{1,1}$ and $\phi\in \Phi^{2}$ satisfies the SND condition (\ref{fiohp1.3}), then we have
 $$\|T_{\phi,a}f\|_{L^1}\leq C\|f\|_{h^1}.$$
 \end{corollary}

  In fact, we establish a more comprehensive theorem in this paper, which generalizes Theorem 1.1. Prior to presenting our theorem, we introduce the rough H\"{o}rmander class $L^{\infty}S^{m}_{\rho}$,
  which is defined by Kenig-Staubach \cite{KS07}. Let $m\in \mathbb{R}$ and $0\leq\rho\leq 1$. A function $a$ that is smooth in the frequency variable $\xi$ and bounded measurable in the spatial variable $x$
  belongs to the rough H{\"o}rmander class $L^{\infty}S^{m}_{\rho}$,  if it satisfies
 \begin{equation}\label{fiohp1.5}
 \sup_{\xi\in \mathbb{R}^n}(1+|\xi|)^{-m+\rho N}\left\|\nabla^N_{\xi}a(\cdot,\xi)\right\|_{L^{\infty}( \mathbb{R}^n)}=A_{N}<\infty
 \end{equation}
 for all $N\in \mathbb{N}$. It is easy to see that $S^m_{\rho,1}\subset L^{\infty}S^{m}_{\rho}$.

 In this note, we prove the following theorem.
 \begin{theorem}
 Suppose that $0\leq\rho\leq 1,n\geq 2$ or $0\leq\rho<1,n=1$, $a\in L^{\infty}S^{m_p}_{\rho}$ where $m_p$ is given by (\ref{fiohp1.4}) and $\phi\in \Phi^{2}$ satisfies the SND condition (\ref{fiohp1.3}).
 Then for any $\frac {n}{n+1}<p\leq 1$, we have
 $$\|T_{\phi,a}f\|_{L^p}\leq C\|f\|_{h^p}.$$
 Furthermore, if $a$ has compact support in variable $x$, then for any $0<p\leq 1$ we have
  $$\|T_{\phi,a}f\|_{L^p}\leq C\|f\|_{h^p}.$$
 The constants here depend only on $n,\rho,\lambda,p$ and finitely many semi-norms $A_{N}, B_{N,M}$.
 \end{theorem}

 \textbf{Remark 1.} We restrict $n,\rho$ just to ensure the $L^2$ boundedness of $T_{\phi,a}$. When $n=\rho=1$, according to \cite[Theorem 1.2]{GZ22}, we know that Theorem 1.1 is false as
 there exists a amplitude $a\in S^{0}_{1,1}$ such that the pseudo-differential operator $T_a$ is unbounded from $H^1$ to $L^1$.

 \textbf{Remark 2.} When $0\leq\rho\leq \frac 12$, the FIO behaves like a PDO and the bound on $m$ in Theorem 1.1 is sharp when $p=1$. When $\rho=1$, the bound on $m$ in Theorem 1.1 is also sharp when $p=1$.
  When $0\leq\rho\leq \frac 12$, one can find a counterexample in \cite{KS07} or \cite{GZ22} for PDOs. When $\rho=1$, one can find a counterexample in \cite[p. 426]{S93} for FIOs.

 \textbf{Remark 3.} Obviously, Theorem 1.3 can be considered as a supplement of Theorem C to $\delta=1$.

 \textbf{Remark 4.} More interestingly, one can check that $m_p>(\rho-n)(\frac 1p-\frac12)+\frac{n}{2}(\rho-\delta)$ when $p\leq 1$ and
 $$1-\frac{n-1}{n}(2\rho_0-\rho)<\delta\leq 1.$$
 As $S^{m}_{\rho,\delta}\subset L^{\infty}S^{m}_{\rho}$ for any $m\in \mathbb{R},0\leq\rho,\delta\leq1$,
 by Theorem 1.1 we can improve Theorem C strictly when $1-\frac{n-1}{n}(2\rho_0-\rho)<\delta<1$.
 By using Fefferman-Stein interpolation, we can also improve Theorem A strictly when $1-\frac{n-1}{n}(2\rho_0-\rho)<\delta<1$ and $1<p<2$. So, it is expected that Theorem A and Theorem C can be improved
 when $\rho<\delta<1$ and $p<2$. We will discuss this problem carefully later.

 We use some new techniques to improve the corresponding estimates. Two of them are particularly important. Firstly, for an $h^1$ atom supported in $B(0,r)$ when $r<1$, we directly utilize the properties of its Fourier transform. This enables us to overcome a critical difficult arising from the part $\sum\limits_{r^{-1}<2^{j}<r^{-\frac {1}{\rho}}}$
 when $p=1$ in Theorem 1.3. Secondly, we introduce a new "except set" $P_r$ instead of the one in Seeger-Sogge-Stein \cite{SSS91}.  Through the properties of $P_r$, we are able to simplify and enhance some crucial computations..

 In Section 2 we introduce some notations, basic inequalities and lemmas. We prove Theorem 1.3 in Section 3 when $p=1$ and Section 4 when $p<1$. In two proofs, we use different decompositions.
 When $p=1,r<1$, with the help of (\ref{fiohp3.5}) we can deal with the term $\sum\limits_{2^{j\rho}r\leq 1}$ directly in Section 3.2.2. However, when $p<1$, we can
 only deal with the term $\sum\limits_{2^{j}r\leq 1}$. Furthermore, when $p<1$, we need the $h^1-L^1$-boundedness instead of corresponding $L^2$-boundedness in (\ref{fiohp4.8}).
 Perhaps one could use a relatively complex uniform proof for $p\leq 1$ with the help of a similar inequality to (\ref{fiohp3.5}). For completeness, we include proofs of some lemmas in the appendices, which have been implicitly proven and used in previous literature but we can not find explicit direct reference.

 Throughout this note, $A\lesssim B$ means that $A\leq CB$ for some constant $C$. The notation $A\approx B$ means that $A\lesssim B$ and $B\lesssim A$. Without explanation,
 the implicit constants given in this note may vary from occasion to occasion but depend only on $n,\rho,p,\lambda$ and finitely many semi-norms of $A_N,B_{N,M}$.

 \section{Some notations and lemmas}

 \hspace{4mm} Let $B_r(x_0)$ be the ball in $\mathbb{R}^n$ centered at $x_0$ with a radius of $r$.

 At first we introduce the theory of atom decomposition of $h^p$ with $0<p\leq 1$.  A function $b$ is called a $L^2$-atom for $h^p(\mathbb{R}^n)$ if\\
  (1) $b$ is supported in $B_r(x_0)$ for some $x_0\in \mathbb{R}^n$;\\
  (2) $\|b\|_2\leq r^{n(\frac 12-\frac 1p)}$;\\
  (3) when $r<1$, $\int_{\mathbb{R}^n}(x-x_0)^{\alpha}b(x)dx=0$ for all multi-indices $\alpha$ with $|\alpha|\leq n(\frac 1p-1)$. When $r\geq 1$, $b$ only needs to satisfy (1) and (2).

 It is also well known that a distribution $f\in h^p(\mathbb{R}^n)$ has an atomic decomposition $f=\sum\limits_{j}\lambda_jb_j,$ with $\sum_{j}|\lambda_j|^p<+\infty$.
 Moreover, there holds $$\|f\|_{h^p(\mathbb{R}^n)}\approx\inf\{(\sum_{j}|\lambda_j|^p)^{\frac 1p}:f=\sum\limits_{j}\lambda_jb_j\}.$$

 We recall a lemma for PDOs that will be used in the proof when $p=1,\rho=0$.
 \begin{lemma}\cite[Theorem 1.2]{GZ22}
 If $a\in L^\infty S^{-n}_{0}$, then $T_a$ is bounded from $h^1$ to $L^1$.
 \end{lemma}
 In fact, Theorem 1.2 in \cite{GZ22} is only proved for $H^1$. However, one can easily verify that the proof of \cite[Theorem 1.2]{GZ22} is also valid for $h^1$.
 For completeness, we include the proof of Lemma 2.1 in Appendix A.

 Secondly, thank to the SND condition (\ref{fiohp1.3}), for any $E\subset \mathbb{R}^n$, $\xi\in S^{n-1}$ and a real-value function $g$, there holds
 \begin{align}
 \int_{\nabla_{\xi}\phi(x,\xi)\in E}\left|g\left(\nabla_{\xi}\phi(x,\xi)\right)\right|dx\lesssim \int_{E}|g(y)|dy.\label{fiohp2.1}
 \end{align}
 This inequality will be used frequently below.

 Thirdly, we introduce the Littlewood-Paley dyadic decomposition and the second dyadic decomposition.

 Take a nonnegative function $\Psi_0\in C^{\infty}_{c}(B_{2})$ such that $\Psi_0\equiv 1$ on $B_{1}$, and
 set $\Psi(\xi)=\Psi_0(\xi)-\Psi_0(2\xi)$. It is easy to see that $\Psi$ is supported in $\{\xi\in \mathbb{R}^n: \frac 12<|\xi|<2\}$ and
 $$\Psi_0(\xi)+\sum^{\infty}_{j=1}\Psi(2^{-j}\xi)=1, \forall \xi\in \mathbb{R}^n.$$
 For the sake of convenience, we denote $\Psi_j(\xi)=\Psi(2^{-j}\xi)$ when $j>0$.

 For every $j>0$, there  are no more than $C2^{j(n-1)\rho_0}$ points $\xi^\nu_j\in \mathbb{S}^{n-1}$ ($\nu=1,2,\ldots, J$ with $J\leq C2^{j(n-1)\rho_0}$) such that
 \begin{align*}
 |\xi^{\nu_1}_j-\xi^{\nu_2}_j|\geq 2^{-j\rho_0-2}\textrm{ if } \nu_1\neq \nu_2 \textrm{ and }
 \inf_{\nu}|\xi^{\nu}_j-\xi|\leq 2^{-j\rho_0}, \forall \xi\in \mathbb{S}^{n-1}.
 \end{align*}

 For $j,\nu>0$, set
 \begin{align*}
 \Gamma^{\nu}_{j}=\{\xi: |\frac{\xi}{|\xi|}-\xi^{\nu}_{j}|\leq 2^{2-j\rho_0}\}\textrm{ and }A^{\nu}_{j}=\Gamma^{\nu}_{j}\cap\{\xi: 2^{j-1}<|\xi|<2^{j+1}\}.
 \end{align*}
 Then by  using the same arguments as in \cite[p. 20-21]{CIS21}, we can construct a partition of unity $\{\psi^{\nu}_{j}\}_{\nu=1}^J$ associated with the family $\{\Gamma^{\nu}_{j}\}_{\nu=1}^J$ for any $j>0$.
 Each $\psi^{\nu}_{j}$ is homogeneous of degree 0, supported in $\Gamma^{\nu}_{j}$, and satisfies that
 \begin{align}
 \sum^{J}_{\nu=1}\psi^{\nu}_{j}(\xi)\equiv1 \textrm{ if }\xi\neq 0,\left|\nabla^{k}_{\xi}\psi^{\nu}_{j}(\xi)\right|\leq C_{k}|\xi|^{-k}2^{jk\rho_0}, k\in\mathbb{N}, \label{fiohp2.2}
 \end{align}
 where $C_k$ depends only on $k$.

 Below we usually use the following notations,
 \begin{align*}
 &a_j(x,\xi)=a(x,\xi)\Psi_j(\xi),T_{\phi,a_j}f(x)=\int_{\mathbb{R}^{n}}e^{i\phi(x,\xi)}a_j(x,\xi)\widehat{f}(\xi)d\xi,j\geq 0;\\
 &T^{\nu}_{j}f(x)=\int_{\mathbb{R}^{n}}e^{i\phi(x,\xi)}a_j(x,\xi)\psi^{\nu}_{j}(\xi)\widehat{f}(\xi)d\xi, j,\nu>0.
 \end{align*}
 One can easily see that $T_{\phi,a_j}f,T^{\nu}_{j}f$ are well defined for any $f\in L^{1}$.

 Set
 $$h^{\nu}_{j}(x,\xi)=\phi(x,\xi)-\xi\cdot \nabla_{\xi}\phi(x,\xi^{\nu}_{j}).$$
 By using similar arguments in \cite[p. 407]{S93}, the following lemma has been implicitly proved in some literatures, such as \cite{CIS21,DGZ23,FS14,SZ23}.
 \begin{lemma}
 If $j>0$ and $\phi\in \Phi^{2}$, then for any $\xi\in A^{\nu}_{j}$ we have
 \begin{align*}
 &|\partial^{N}_{\xi^{\nu}_{j}}\nabla^{M}_{\xi}h^{\nu}_{j}|\lesssim 2^{-j(N\rho+M\rho_0)}, \textrm{if $N,M\geq 0$ with $N+M\geq 1$};\\
 &|\partial^{N}_{\xi^{\nu}_{j}}\nabla^{M}_{\xi}\psi^{\nu}_{j}|\lesssim 2^{-j(N+M(1-\rho_0))}\leq 2^{-j(N\rho+M\rho_0)}, \textrm{if $N,M\geq 0$}.
 \end{align*}
 The constants depend only on $n,\rho,N,M$ and finitely many semi-norms of $\phi\in \Phi^{2}$.
 \end{lemma}
 For completeness, we include the proof of Lemma 2.2 in Appendix B.

 As a direct corollary, we give the following lemma.
 \begin{lemma}
 If $N,M\geq 0,j>0$, $a\in L^{\infty}S^{m}_{\rho}$ and $\phi\in \Phi^{2}$, then for any $\xi\in A^{\nu}_{j}$ there holds
 $$\left|\partial^{N}_{\xi^{\nu}_{j}}\nabla^{M}_{\xi}(e^{ih^{\nu}_{j}}a_j\psi^{\nu}_{j})\right|\leq C2^{j(m-N\rho-M\rho_0)},$$
 where $C$ depends only on $n,m,\rho,N,M$ and finitely many semi-norms of $a\in L^{\infty}S^{m}_{\rho}$ and $\phi\in \Phi^{2}$.
 \end{lemma}
 \begin{proof}
 By Lemma 2.2 and the product rule, for any $k,l\geq 0,j>0$ we can get that
 \begin{align*}
 |\partial^{k}_{\xi^{\nu}_{j}}\nabla^{l}_{\xi}e^{ih^{\nu}_{j}}|
 \lesssim &\sum^{k+l}_{t=1}\sum_{\begin{subarray}{c}k_1+\cdots+k_t=k,l_1+\cdots+l_t=l\\k_1+l_1,\ldots,k_t+l_t>0\end{subarray}}
 |\partial^{k_1}_{\xi^{\nu}_{j}}\nabla^{l_1}_{\xi}h^{\nu}_{j}\cdots \partial^{k_1}_{\xi^{\nu}_{j}}\nabla^{l_1}_{\xi}h^{\nu}_{j}|\nonumber\\
 \lesssim &\sum^{k+l}_{t=1}\sum_{\begin{subarray}{c}k_1+\cdots+k_t=k,l_1+\cdots+l_t=l\\k_1+l_1,\ldots,k_t+l_t>0\end{subarray}}
 2^{-j(k_1\rho+l_1\rho_0)}\cdots 2^{-j(k_t\rho+l_t\rho_0)}\nonumber\\
 \lesssim &2^{-j(k\rho+l\rho_0)}.
 \end{align*}
 Similarly, thank to the facts $a\in L^{\infty}S^{m}_{\rho}$ and $\rho_0\leq \rho$, by Lemma 2.2 we get that
 \begin{align*}
 |\partial^{k}_{\xi^{\nu}_{j}}\nabla^{l}_{\xi}(a_j\psi^{\nu}_{j})|
 \lesssim &\sum_{k_1+k_2=k}\sum_{l_1+l_2=l}
 |\nabla^{k_1+l_1}_{\xi}a_j| |\partial^{k_2}_{\xi^{\nu}_{j}}\nabla^{l_2}_{\xi}\psi^{\nu}_{j}|\nonumber\\
 \lesssim &\sum_{k_1+k_2=k}\sum_{l_1+l_2=l}2^{j(m-(k_1+l_1)\rho)}2^{-j(k_2\rho+l_2\rho_0)}\lesssim 2^{j(m-k\rho-l\rho_0)}.
 \end{align*}
 Therefore, by the chain rule, we show that
 \begin{align*}
 |\partial^{N}_{\xi^{\nu}_{j}}\nabla^{M}_{\xi}(e^{ih^{\nu}_{j}}a_j\psi^{\nu}_{j})|\lesssim \sum_{N_1+N_2=N}\sum_{M_1+M_2=M}2^{-j(N_1\rho+M_1\rho_0)}2^{j(m-N_2\rho-M_2\rho_0)}=2^{j(m-N\rho-M\rho_0)}.
 \end{align*}
 This finishes the proof.
 \end{proof}

 Fourthly, we define the "except set" and give some basic estimates.

 When $r\geq 1$, it is very simple. We set
 $$\widetilde{B}_{r}=\bigcup_{\xi\in S^{n-1}}\{x:|\nabla_{\xi}\phi(x,\xi)|\leq 3r\}.$$
 Take a $\xi_0\in S^{n-1}$. As $\phi\in \Phi^2$, for any $\xi\in S^{n-1}$ it holds
 $$|\nabla_{\xi}\phi(x,\xi_0)-\nabla_{\xi}\phi(x,\xi)|\leq \pi\sup_{\eta\in S^{n-1}}|\nabla^2_{\xi}\phi(x,\eta)|\leq \pi B_{2,0}.$$
 For any $x\in \widetilde{B}_{r}$, there exists a $\xi\in S^{n-1}$ such that $|\nabla_{\xi}\phi(x,\xi)|\leq 3r$. So we have
 $$|\nabla_{\xi}\phi(x,\xi_0)|\leq 3r+\pi B_{2,0}\leq (3+\pi B_{2,0})r$$
 which yields that
 $$\widetilde{B}_{r}\subset \{x:|\nabla_{\xi}\phi(x,\xi_0)|\leq (3+\pi B_{2,0})r\}.$$
 From (\ref{fiohp2.1}), we get that
 \begin{align}
 |\widetilde{B}_{r}|\lesssim |\{z:|z|\leq (3+\pi B_{2,0})r\}|\lesssim r^n.\label{fiohp2.3}
 \end{align}

 When $0<r<1$, we set
 $$P_{r}=\bigcup_{\xi\in S^{n-1}}\{x: |\phi(x,\xi)|\leq 3r \textrm{ and }  |\nabla_{\xi}\phi(x,\xi)|\leq 3r^{\lambda_{\rho}}\}.$$
 For the set $P_r$, we can get the following lemma.
 \begin{lemma}
 If $r<1$ and $\phi\in \Phi^{2}$ satisfies the SND condition (\ref{fiohp1.3}), then there holds
 \begin{align}
 |P_{r}|\lesssim r.\label{fiohp2.4}
 \end{align}
 On the other hand, for any $x\notin P_{r}$, $|y|<r$, $j>0$ and $\xi\in S^{n-1}$, we can get that
 \begin{align}
 &1+2^{j\rho}|\xi\cdot\nabla_{\xi}\phi(x,\xi)|+2^{j\rho_0}|\nabla_{\xi}\phi(x,\xi)|\nonumber\\
 \lesssim &(1+2^{j\rho}|\xi\cdot(\nabla_{\xi}\phi(x,\xi)-y)|+2^{j\rho_0}|\nabla_{\xi}\phi(x,\xi)-y|)^{\frac{1}{\lambda_{\rho}}}.\label{fiohp2.5}
 \end{align}
 The constants here depend only on $n,\lambda$ and finitely many semi-norms of $\phi\in \Phi^{2}$.
 \end{lemma}
 \begin{proof} Let $j_1$ be a positive integer such that $1\leq 2^{j_1\rho_0}r^{\lambda_{\rho}}=(2^{j_1\rho}r)^{\lambda_{\rho}}<2$. If $x\in P_r$, then there exists a $\xi_0\in S^{n-1}$ such that
 $$|\phi(x,\xi_0)|\leq 3r \textrm{ and }  |\nabla_{\xi}\phi(x,\xi_0)|\leq 3r^{\lambda_{\rho}}.$$
 In the second dyadic decomposition, we can find $\nu$ such that $|\xi_0-\xi^{\nu}_{j_1}|\leq 2^{-j_1\rho_0}\leq r^{\lambda_{\rho}}$. As $\phi\in \Phi^2$, when
 $\xi\in S^{n-1}$ and $|\xi-\xi_0|\leq r^{\lambda_{\rho}}$, we have
 \begin{align*}
 &|\nabla_{\xi}\phi(x,\xi)|\leq |\nabla_{\xi}\phi(x,\xi_0)|+|\nabla_{\xi}\phi(x,\xi)-\nabla_{\xi}\phi(x,\xi_0)|\\
 \leq &3r^{\lambda_{\rho}}+2|\xi-\xi_0|\sup_{\eta\in S^{n-1}}|\nabla^2_{\xi}\phi(x,\eta)|\leq (3+2B_{2,0})r^{\lambda_{\rho}}.
 \end{align*}
 As $r<1,\frac 12\leq \lambda_{\rho}\leq 1$, by the above inequality we obtain that
 \begin{align*}
 &|\phi(x,\xi^{\nu}_{j_1})|\leq |\phi(x,\xi_0)|+|\phi(x,\xi_0)-\phi(x,\xi^{\nu}_{j_1})|\\
 \leq &3r+2|\xi_0-\xi^{\nu}_{j_1}|\sup_{\xi\in S^{n-1},|\xi-\xi_0|\leq r^{\lambda_{\rho}}}|\nabla_{\xi}\phi(x,\xi)|\\
 \leq &3r+2r^{\lambda_{\rho}}(3+2B_{2,0})r^{\lambda_{\rho}}\leq (9+4B_{2,0})r.
 \end{align*}
 Thus we show that
 $$|\phi(x,\xi^{\nu}_{j_1})|\leq (9+4B_{2,0})r,|\nabla_{\xi}\phi(x,\xi^{\nu}_{j_1})|\leq (3+2B_{2,0})r^{\lambda_{\rho}}.$$
 Since $\phi$ is positively homogeneous of degree 1, it is easy to see that $\phi(x,\xi^{\nu}_{j_1})=\xi^{\nu}_{j_1}\cdot \nabla_{\xi}\phi(x,\xi^{\nu}_{j_1})$. Now we get that
 $$P_{r}\subset \bigcup^{C2^{j_1(n-1)\rho_0}}_{\nu=1}\{x:|\xi^{\nu}_{j_1}\cdot \nabla_{\xi}\phi(x,\xi^{\nu}_{j_1})|\leq (9+4B_{2,0})r \textrm{ and } |\nabla_{\xi}\phi(x,\xi^{\nu}_{j_1})|\leq (3+2B_{2,0})r^{\lambda_{\rho}}\}.$$

 Thank to (\ref{fiohp2.1}) and $1\leq 2^{j_1\rho_0}r^{\lambda_{\rho}}<2$, we get that
 \begin{align*}
 |P_{r}|\leq &\sum^{C2^{j_1(n-1)\rho_0}}_{\nu=1}\left|\{x:|\xi^{\nu}_{j_1}\cdot \nabla_{\xi}\phi(x,\xi^{\nu}_{j_1})|\leq (9+4B_{2,0})r \textrm{ and }
 |\nabla_{\xi}\phi(x,\xi^{\nu}_{j_1})|\leq (3+2B_{2,0})r^{\lambda_{\rho}}\}\right|\\
 \lesssim  &\sum^{C2^{j_1(n-1)\rho_0}}_{\nu=1}\left|\{x:|\xi^{\nu}_{j_1}\cdot x|\leq (9+4B_{2,0})r \textrm{ and }|x|\leq (3+2B_{2,0})r^{\lambda_{\rho}}\}\right|\\
 \lesssim  &2^{j_1(n-1)\rho_0}r^{1+(n-1)\lambda_{\rho}}\lesssim r.
 \end{align*}
 This finishes the proof of (\ref{fiohp2.4}).

 It is easy to see that $1+2^{j\rho}r+2^{j\rho_0}r^{\lambda_{\rho}}\leq 2(1+2^{j\rho}r)$ as $\lambda_{\rho}\leq 1$ and $\rho_0=\lambda_{\rho}\rho$. So, when $|y|<r,j>0$ and $\xi\in S^{n-1}$,
 to prove (\ref{fiohp2.5}), it is enough for us to show that
 $$2^{j\rho}r\lesssim (1+2^{j\rho}|\xi\cdot(\nabla_{\xi}\phi(x,\xi)-y)|+2^{j\rho_0}|\nabla_{\xi}\phi(x,\xi)-y|)^{\frac{1}{\lambda_{\rho}}}.$$
 As $\phi$ is positively homogeneous of degree 1, for any $x\notin P_r$ and $\xi\in S^{n-1}$, there must be
 $$|\xi\cdot\nabla_{\xi}\phi(x,\xi)|=|\phi(x,\xi)|>3r\textrm{ or }|\nabla_{\xi}\phi(x,\xi)|>3r^{\lambda_{\rho}}.$$

 If $|\xi\cdot\nabla_{\xi}\phi(x,\xi)|>3r$, due to the fact $\lambda_{\rho}\leq 1$, we have
 $$(1+2^{j\rho}|\xi\cdot(\nabla_{\xi}\phi(x,\xi)-y)|+2^{j\rho_0}|\nabla_{\xi}\phi(x,\xi)-y|)^{\frac{1}{\lambda_{\rho}}}\geq (1+2^{j\rho}r)^{\frac{1}{\lambda_{\rho}}}\geq 2^{j\rho}r.$$
 Otherwise, if $|\nabla_{\xi}\phi(x,\xi)|>3r^{\lambda_{\rho}}$, as $\lambda_{\rho}\leq 1$ and $\rho_0=\lambda_{\rho}\rho$, we get that
 \begin{align*}
 (1+2^{j\rho}|\xi\cdot(\nabla_{\xi}\phi(x,\xi)-y)|+2^{j\rho_0}|\nabla_{\xi}\phi(x,\xi)-y|)^{\frac{1}{\lambda_{\rho}}}
 \geq (1+2^{j\rho_0}r^{\lambda_{\rho}})^{\frac{1}{\lambda_{\rho}}}\geq2^{j\rho}r.
 \end{align*}
 This finishes the proof of (\ref{fiohp2.5}).
 \end{proof}

 The following lemma has been implicitly proved and used in many literatures.
 \begin{lemma}(\cite[Lemma 3.1]{RZ23})
  Let $n\geq 1,0\leq\rho\leq 1$, $a\in L^{\infty}S^{m}_{\rho}$ and $\phi\in \Phi^{2}$ satisfies the SND condition (\ref{fiohp1.3}). Then for any $j\in\mathbb{N}$, there holds
 $$\|T_{\phi,a_j}f\|_2\lesssim 2^{j(m-\frac {n(\rho-1)}{2})}\|f\|_2,$$
 where the constant depends only on $n,m,\rho,\lambda$ and finitely many semi-norms $A_{N}, B_{N,M}$.
 \end{lemma}
 Lemma 3.1 in \cite{RZ23} has been proved only for $a\in S^{m}_{\rho,1}$.
 In Appendix C we use an identical proof to establish this lemma.

 At last, we recall a lemma for the fractional integration.
 \begin{lemma}\cite[Corollary 2.3]{K82}\label{l2.1}
 Let $s>0$, $\frac 1p=1+\frac sn$ and the fractional integration $I_s$ is defined as $\widehat{I_{s}(f)}(\xi)=|\xi|^{-s}\widehat{f}(\xi)$. Then $I_{s}$ is bounded from $H^{p}(\mathbb{R}^n)$ to $H^1(\mathbb{R}^n)$.
 \end{lemma}

\section{Proof of Theorem 1.3 when $p=1$}

\subsection{Case: $\rho=0$}

 Take any $\xi_0\in S^{n-1}$. By the SND condition of $\phi$, it is easy to check that $F: x\to \nabla_{\xi}\phi(x,\xi_0)$ is a invertible map from $\mathbb{R}^n$ to $\mathbb{R}^n$.
 So, we can define $\widetilde{a}$ by $$\widetilde{a}(\nabla_{\xi}\phi(x,\xi_0),\xi)=e^{i[\phi(x,\xi)-\nabla_{\xi}\phi(x,\xi_0)\cdot \xi]}a(x,\xi)(1-\Psi_0(\xi)).$$
 By Lemma 2.3 it is easy to see that $\widetilde{a} \in L^{\infty}S^{-n}_{0}$ if $a\in L^{\infty}S^{-n}_{0}$.
 We divide $T_{\phi,a}f$ into two parts as
 \begin{align*}
 T_{\phi,a}f(x)=&T_{\phi,a_0}f(x)+\int_{\mathbb{R}^{n}}e^{i\phi(x,\xi)}a(x,\xi)(1-\Psi_0(\xi))\widehat{f}(\xi)d\xi\\
 =&T_{\phi,a_0}f(x)+\int_{\mathbb{R}^{n}}e^{i\nabla_{\xi}\phi(x,\xi_0)\cdot \xi}e^{i[\phi(x,\xi)-\nabla_{\xi}\phi(x,\xi_0)\cdot \xi]}a(x,\xi)(1-\Psi_0(\xi))\widehat{f}(\xi)d\xi\\
 =&T_{\phi,a_0}f(x)+T_{\widetilde{a}}f(\nabla_{\xi}\phi(x,\xi_0)).
 \end{align*}
  Here $T_{\widetilde{a}}$ is the pseudo-differential operator with amplitude $\widetilde{a} \in L^{\infty}S^{-n}_{0}$. Thus, by Theorem C, (\ref{fiohp2.1}) and Lemma 2.1, we get that
 \begin{align*}
 \|T_{\phi,a}f\|_1\leq \|T_{\phi,a_0}f\|_1+\|T_{\widetilde{a}}f(\nabla_{\xi}\phi(\cdot,\xi_0))\|_1\lesssim \|f\|_{h^1}+\|T_{\widetilde{a}}f\|_1\lesssim \|f\|_{h^1}.
 \end{align*}
 This finishes the proof of Theorem 1.3 when $p=1,\rho=0$.

\subsection{Case: $0<\rho\leq 1$}
 Due to the theory of atom decomposition, it suffices to show $\|T_{\phi,a} b\|_1\lesssim 1$ for any $L^2$-atom $b$ for $h^1(\mathbb{R}^n)$ which is defined in Section 2.

 \subsubsection{Case: $r\geq 1$}

 Let $\widetilde{B}_{r}$ be the one defined in section 2. It is easy to see that $m_1=\rho+(n-1)\rho_0-n<\frac n2(\rho-1)$ when $0<\rho\leq 1,n\geq 2$ or $0<\rho<1,n=1$.
 Therefore, we can use the $L^2$ boundedness of $T_{\phi,a}$ \cite[Theorem 2.2]{FS14} and (\ref{fiohp2.3}) to get
 \begin{align}
 \|T_{\phi,a} b\|_{L^1(\widetilde{B}_{r})}\lesssim |\widetilde{B}_{r}|^{\frac 12}\|T_{\phi,a}b\|_2\lesssim r^{\frac n2}\|b\|_2\lesssim 1.\label{fiohp3.1}
 \end{align}
 By Theorem C we have
 \begin{align}
 \|T_{\phi,a_0} b\|_1\lesssim \|b\|_{h^1}\lesssim 1.\label{fiohp3.2}
 \end{align}

 Take an integer $N>\frac {n+1}{2}$. By integration by parts, when $j>0$ we get
 \begin{align*}
 &|T_{\phi,a_j} b(x)|=\left|\int_{\mathbb{R}^{n}}\int_{|y|<r}e^{i[\phi(x,\xi)-y\cdot \xi]}a_j(x,\xi) b(y)d\xi dy\right|\\
 \leq&\sum^{J}_{\nu=1}\left|\int_{|y|<r}\left(\int_{\mathbb{R}^{n}}e^{i[\nabla_{\xi}\phi(x,\xi^{\nu}_j)-y]\cdot \xi}e^{ih_j^{\nu}(x,\xi)} a_{j}(x,\xi)\psi^{\nu}_{j}(\xi)d\xi\right) b(y)dy\right|\\
 \lesssim &\sum^{J}_{\nu=1}\int_{|y|<r}|\nabla_{\xi}\phi(x,\xi^{\nu}_j)-y|^{-2N}\int_{\mathbb{R}^{n}}\left|\nabla^{2N}_{\xi}[e^{ih_j^{\nu}(x,\xi)}a_j(x,\xi)\psi^{\nu}_{j}(\xi)]\right|d\xi |b(y)|dy.
 \end{align*}

 Thus, when $x\notin \widetilde{B}_{r},j>0$, by Lemma 2.3 we get that
 \begin{align}
 |T_{\phi,a_j} b(x)|\lesssim &\sum^{J}_{\nu=1}\int_{|y|<r}|\nabla_{\xi}\phi(x,\xi^{\nu}_j)|^{-2N}\int_{A^{\nu}_{j}}\left|\nabla^{2N}_{\xi}[e^{ih_j^{\nu}(x,\xi)}a_j(x,\xi)\psi^{\nu}_{j}(\xi)]\right|d\xi |b(y)|dy\nonumber\\
  \lesssim &\sum^{J}_{\nu=1}|\nabla_{\xi}\phi(x,\xi^{\nu}_j)|^{-2N}2^{j(m_1-2N\rho_0)}|A^{\nu}_{j}|\|b\|_1\nonumber\\
 \lesssim &2^{j(\rho-2N\rho_0)}\sum^{J}_{\nu=1}|\nabla_{\xi}\phi(x,\xi^{\nu}_j)|^{-2N}.\label{fiohp3.3}
 \end{align}
 Thank to $0<\rho\leq 2\rho_0, r\geq1, 2N>n+1$, it is easy to see that $\rho+(n-1)\rho_0-2N\rho_0\leq (n+1-2N)\rho_0<0$. So, by (\ref{fiohp3.1}), (\ref{fiohp3.2}),
 (\ref{fiohp3.3}) and (\ref{fiohp2.1}), we can get that
 \begin{align*}
 \|T_{\phi,a} b\|_1\leq  &\|T_{\phi,a} b\|_{L^1(\widetilde{B}_{r})}+\|T_{\phi,a_0} b\|_{L^1(\widetilde{B}^c_{r})}+\sum^{\infty}_{j=1}\|T_{\phi,a_j} b\|_{L^1(\widetilde{B}^c_{r})}\\
 \lesssim &1+\sum^{\infty}_{j=1}2^{j(\rho-2N\rho_0)}\sum^{J}_{\nu=1}\int_{|\nabla_{\xi}\phi(x,\xi^{\nu}_j)|>3r}|\nabla_{\xi}\phi(x,\xi^{\nu}_j)|^{-2N}dx\\
 \lesssim &1+\sum^{\infty}_{j=1}2^{j(\rho+(n-1)\rho_0-2N\rho_0)}r^{n-2N}\lesssim 1.
 \end{align*}
 Therefore, we proved that $\|T_{\phi,a}b\|_1\lesssim 1$ if $r\geq1$.

 \subsubsection{Case: $r<1$ and $2^{j\rho}r<1$}

 We first notice two basic estimates for the atom $b$ when $r<1$. For any $N\in \mathbb{N}$, there holds
 \begin{align}
 |\nabla^N\widehat{b}(\xi)|\leq \int_{|y|\leq r}|y|^N|b(y)|\,\textrm{d}y\lesssim r^N.\label{fiohp3.4}
 \end{align}
 Also, \cite[Section III.5.4]{S93} gives
 \begin{align}
 \int_{\mathbb{R}^n}|\widehat{b}(\xi)||\xi|^{-n}\,\textrm{d}\xi\lesssim \|b\|_{H^1}\lesssim 1.\label{fiohp3.5}
 \end{align}

 Set
 $$L=1+2^{2j\rho}\partial^2_{\xi^{\nu}_{j}}+2^{2j\rho_0}\sum\limits^n_{k=1}\partial^2_{\xi_k}.$$
 It is to see that $L$ is self-adjoint. Then by (\ref{fiohp3.4}) and Lemma 2.3, when $2^{j\rho}r<1$, for any $N\in \mathbb{N}$ we can show that
 \begin{align*}
 &|T^{\nu}_{j}b(x)|=|\int_{\mathbb{R}^{n}}e^{i\nabla_{\xi}\phi(x,\xi^{\nu}_j)\cdot \xi}e^{ih_j^{\nu}(x,\xi)} a_j(x,\xi)\psi^{\nu}_{j}(\xi)\widehat{b}(\xi)d\xi|\\
 =&(1+2^{2j\rho}|\partial_{\xi^{\nu}_{j}}\phi(x,\xi^{\nu}_j)|^2+2^{2j\rho_0}|\nabla_{\xi}\phi(x,\xi^{\nu}_j)|^2)^{-N}\\
 &|\int_{\mathbb{R}^{n}}L^N(e^{i\nabla_{\xi}\phi(x,\xi^{\nu}_j)\cdot \xi})e^{ih_j^{\nu}(x,\xi)} a_j(x,\xi)\psi^{\nu}_{j}(\xi)\widehat{b}(\xi)d\xi|\\
 \lesssim &(1+2^{2j\rho}|\partial_{\xi^{\nu}_{j}}\phi(x,\xi^{\nu}_j)|^2+2^{2j\rho_0}|\nabla_{\xi}\phi(x,\xi^{\nu}_j)|^2)^{-N}
 \int_{\mathbb{R}^{n}}|L^N(e^{ih_j^{\nu}} a_j\psi^{\nu}_{j}\widehat{b})|d\xi\\
 \lesssim &(1+2^{j\rho}|\partial_{\xi^{\nu}_{j}}\phi(x,\xi^{\nu}_j)|+2^{j\rho_0}|\nabla_{\xi}\phi(x,\xi^{\nu}_j)|)^{-2N}\\
 &\sum_{k_1+k_2+k_3+k_4\leq 2N}\int_{\mathbb{R}^{n}}2^{jk_1\rho}2^{jk_2\rho_0}|\partial^{k_1}_{\xi^{\nu}_{j}}\nabla^{k_2}_{\xi}(e^{ih_j^{\nu}}a_j\psi^{\nu}_{j})|
 2^{jk_3\rho}2^{jk_4\rho_0}|\nabla^{k_3+k_4}_{\xi}\widehat{b}|d\xi\\
 \lesssim &(1+2^{j\rho}|\partial_{\xi^{\nu}_{j}}\phi(x,\xi^{\nu}_j)|+2^{j\rho_0}|\nabla_{\xi}\phi(x,\xi^{\nu}_j)|)^{-2N}
 \sum_{k_3+k_4\leq 2N}\int_{A^{\nu}_{j}}2^{jm_1}2^{j(k_3+k_4)\rho}|\nabla^{k_3+k_4}_{\xi}\widehat{b}|d\xi\\
 \lesssim &2^{jm_1}(1+2^{j\rho}|\partial_{\xi^{\nu}_{j}}\phi(x,\xi^{\nu}_j)|+2^{j\rho_0}|\nabla_{\xi}\phi(x,\xi^{\nu}_j)|)^{-2N}
 (\int_{A^{\nu}_{j}}|\widehat{b}(\xi)|d\xi+2^{j\rho}r|A^{\nu}_{j}|).
 \end{align*}
 Thus, for an integer $N>n/2$, from (\ref{fiohp2.1}) and (\ref{fiohp3.5}) we have
 \begin{align}
 &\sum_{2^{j\rho}r<1}\int_{\mathbb{R}^n}|T_{\phi,a_j} b(x)|dx\leq \sum_{2^{j\rho}r<1}\sum^{J}_{\nu=1}\int_{\mathbb{R}^n}|T^{\nu}_{j}b(x)|dx\nonumber\\
 \lesssim &\sum_{2^{j\rho}r<1}2^{jm_1}\sum^{J}_{\nu=1}(\int_{A^{\nu}_{j}}|\widehat{b}(\xi)|d\xi+2^{j\rho}r|A^{\nu}_{j}|)
 \int_{\mathbb{R}^{n}}(1+2^{j\rho}|\partial_{\xi^{\nu}_{j}}\phi(x,\xi^{\nu}_j)|+2^{j\rho_0}|\nabla_{\xi}\phi(x,\xi^{\nu}_j)|)^{-2N}dx\nonumber\\
 \lesssim &\sum_{2^{j\rho}r<1}\sum^{J}_{\nu=1}2^{-jn}(\int_{A^{\nu}_{j}}|\widehat{b}(\xi)|d\xi+2^{j\rho}r|A^{\nu}_{j}|)\nonumber\\
 \lesssim &\int_{\mathbb{R}^{n}}|\widehat{b}(\xi)||\xi|^{-n}d\xi+\sum_{2^{j\rho}r<1}2^{j\rho}r\lesssim 1.\label{fiohp3.6}
 \end{align}
 This finishes the proof.

 \subsubsection{Case: $r<1$ and $2^{j\rho}r\geq 1$}

 When $r<1$, let $P_r$ be the "except set" defined in section 2. By Lemma 2.5 we get
 \begin{align}
 \|T_{\phi,a_j} b\|_2\lesssim 2^{j(m_1-\frac{n(\rho-1)}{2})}\|b\|_2\lesssim 2^{j(m_1-\frac{n(\rho-1)}{2})}r^{-\frac n2}.\label{fiohp3.7}
 \end{align}

  When $0<\rho\leq 1,n\geq 2$ or $0<\rho<1,n=1$, it is easy to check that $m_1-\frac{n(\rho-1)}{2}<0$ and
 $n-\rho-2(n-1)\rho_0\geq 0$. Thus, from (\ref{fiohp2.4}) and (\ref{fiohp3.7}) we can show that
 \begin{align}
 \sum_{2^{j\rho}r\geq 1}\|T_{\phi,a_j} b\|_{L^1(P_r)}&\leq \sum_{2^{j\rho}r\geq 1}|P_r|^{\frac 12}\|T_{\phi,a_j} b\|_2\nonumber\\
 \lesssim &\sum_{2^{j\rho}r\geq 1}2^{j(m_1-\frac{n(\rho-1)}{2})}r^{\frac 12-\frac n2}\nonumber\\
 \lesssim &r^{-\frac{m_1}{\rho}+\frac{n(\rho-1)}{2\rho}+\frac 12-\frac n2}\nonumber\\
 \lesssim &r^{\frac{n-\rho-2(n-1)\rho_0}{2\rho}}\lesssim 1.\label{fiohp3.8}
 \end{align}

 On the other hand, let $L$ be the operator defined in section 3.2.2. For any $N\in \mathbb{N}$ and $x\notin P_r$, by using (\ref{fiohp2.5}) and Lemma 2.3, we can deduce that
 \begin{align*}
 |T^{\nu}_{j} b(x)|=&|\int_{|y|<r}\int_{\mathbb{R}^{n}}e^{i(\nabla_{\xi}\phi(x,\xi^{\nu}_j)-y)\cdot \xi}e^{ih_j^{\nu}(x,\xi)} a_j(x,\xi)\psi^{\nu}_{j}(\xi)d\xi b(y)dy|\\
 = &|\int_{|y|<r}(1+2^{2j\rho}|\xi^{\nu}_j\cdot (\nabla_{\xi}\phi(x,\xi^{\nu}_j)-y)|^2+2^{2j\rho_0}|\nabla_{\xi}\phi(x,\xi^{\nu}_j)-y|^2)^{-N}\\
 &\int_{\mathbb{R}^{n}}L^N(e^{i(\nabla_{\xi}\phi(x,\xi^{\nu}_j)-y)\cdot \xi})e^{ih_j^{\nu}(x,\xi)} a_j(x,\xi)\psi^{\nu}_{j}(\xi)d\xi b(y)dy|\\
 \lesssim &\int_{|y|<r}(1+2^{j\rho}|\xi^{\nu}_j\cdot (\nabla_{\xi}\phi(x,\xi^{\nu}_j)-y)|+2^{j\rho_0}|\nabla_{\xi}\phi(x,\xi^{\nu}_j)-y|)^{-2N}|b(y)|dy\\
 &\cdot \int_{\mathbb{R}^{n}}|L^N(e^{ih_j^{\nu}} a_j\psi^{\nu}_{j})|d\xi \\
 \lesssim &(1+2^{j\rho}|\xi^{\nu}_j\cdot \nabla_{\xi}\phi(x,\xi^{\nu}_j)|+2^{j\rho_0}|\nabla_{\xi}\phi(x,\xi^{\nu}_j)|)^{-2\lambda_{\rho}N}\|b\|_1\\
 &\cdot \sum_{k_1+k_2\leq 2N}\int_{A^{\nu}_j}2^{jk_1\rho}2^{jk_2\rho_0}|\partial^{k_1}_{\xi^{\nu}_{j}}\nabla^{k_2}_{\xi}(e^{ih_j^{\nu}}a_j\psi^{\nu}_{j})|d\xi\\
 \lesssim &2^{jm_1}|A^{\nu}_j|(1+2^{j\rho}|\xi^{\nu}_j\cdot \nabla_{\xi}\phi(x,\xi^{\nu}_j)|+2^{j\rho_0}|\nabla_{\xi}\phi(x,\xi^{\nu}_j)|)^{-2\lambda_{\rho}N}.
 \end{align*}

 When $\lambda_{\rho}N>n+1$ and $2^{j\rho}r>1$, by (\ref{fiohp2.1}), the integral of $T_{\phi,a_j} b$ on $P^c_r$ can be controlled by
 \begin{align*}
 &\|T_{\phi,a_j} b\|_{L^1(P^c_r)}\leq \sum^{J}_{\nu=1}\|T^{\nu}_{j} b\|_{L^1(P^c_r)}\\
 \lesssim &2^{jm_1}\sum^{J}_{\nu=1}|A^{\nu}_j|\int_{P^c_r}(1+2^{j\rho}|\xi^{\nu}_j\cdot \nabla_{\xi}\phi(x,\xi^{\nu}_j)|+2^{j\rho_0}|\nabla_{\xi}\phi(x,\xi^{\nu}_j)|)^{-2\lambda_{\rho}N}dx\\
 \lesssim &2^{j\rho}\sum^{J}_{\nu=1}\int_{\{z:|\xi^{\nu}_j\cdot z|>3r\}\cup \{z:|z|>3r^{\lambda_{\rho}}\}}(1+2^{j\rho}|\xi^{\nu}_j\cdot z|+2^{j\rho_0}|z|)^{-2\lambda_{\rho}N}dz\\
 \lesssim &2^{j\rho}\sum^{J}_{\nu=1}\left(\int_{|\xi^{\nu}_j\cdot z|>r}+\int_{|z-(\xi^{\nu}_j\cdot z)\xi^{\nu}_j|>r^{\lambda_{\rho}}}\right)
 (1+2^{j\rho}|\xi^{\nu}_j\cdot z|)^{-\lambda_{\rho}N}(1+2^{j\rho_0}|z-(\xi^{\nu}_j\cdot z)\xi^{\nu}_j|)^{-\lambda_{\rho}N}dz\\
 \lesssim &2^{j\rho}2^{j(n-1)\rho_0}2^{-j\rho}2^{-j(n-1)\rho_0}\left((2^{j\rho}r)^{1-\lambda_{\rho}N}+(2^{j\rho_0}r^{\lambda_{\rho}})^{n-1-\lambda_{\rho}N}\right)\\
 =&(2^{j\rho}r)^{1-\lambda_{\rho}N}+(2^{j\rho}r)^{\lambda_{\rho}(n-1-\lambda_{\rho}N)}\lesssim (2^{j\rho}r)^{-1}
 \end{align*}
 which yields that
 \begin{align}
 \sum_{2^{j\rho}r>1}\|T_{\phi,a_j} b\|_{L^1(P^c_r)}\lesssim \sum_{2^{j\rho}r>1}(2^{j\rho}r)^{-1}\lesssim 1.\label{fiohp3.9}
 \end{align}
 Let everything together, this finishes the proof of Theorem 1.3 when $p=1$.

\section{Proof of Theorem 1.3 when $p<1$}
 Similarly, it is enough for us to prove that $\|T_{\phi,a} b\|_p\lesssim 1$ for any $L^2$-atom $b$ for $h^p(\mathbb{R}^n)$ which is defined in section 2. When $0<p\leq 1$, we recall a basic inequality,
 $$\|\sum^{\infty}_{i=1} f_i\|^p_p=\int_{\mathbb{R}^n}|\sum^{\infty}_{i=1} f_i(x)|^pdx\leq \int_{\mathbb{R}^n}\sum^{\infty}_{i=1}|f_i(x)|^pdx=\sum^{\infty}_{i=1}\|f_i\|^p_p.$$

 \subsection{Case: $r\geq 1$}

 Let $\widetilde{B}_{r}$ be the one defined in section 2. When $0\leq \rho\leq 1,n\geq 2$ or $0\leq\rho<1,n=1$, it is easy to see that $m_p=\frac{\rho-n}{p}+(n-1)\rho_0<\frac n2(\rho-1).$
 Therefore, by using (\ref{fiohp2.3}) and the $L^2$ boundedness of $T_{\phi,a}$ \cite[Theorem 2.2]{FS14} we can get that
 \begin{align}
 \left\|T_{\phi,a} b\right\|_{L^p(\widetilde{B}_{r})}\lesssim |\widetilde{B}_{r}|^{\frac 1p-\frac 12}\|T_{\phi,a}b\|_2\lesssim r^{n(\frac 1p-\frac 12)}\|b\|_2\lesssim 1.\label{fiohp4.1}
 \end{align}

 Take an integer $N$ such that $2Np>n+1$. By integration by parts and Lemma 2.3, for any $j,\nu>0$ and $x\notin \widetilde{B}_{r}$, we obtain that
 \begin{align}
 \left|T^{\nu}_{j} b(x)\right|=&\left|\int_{|y|<r}\int_{\mathbb{R}^{n}}e^{i[\nabla_{\xi}\phi(x,\xi^{\nu}_j)-y]\cdot \xi}e^{ih_j^{\nu}(x,\xi)} a_j(x,\xi)\psi^{\nu}_{j}(\xi)b(y)d\xi dy\right|\nonumber\\
 \lesssim &\int_{|y|<r}|\nabla_{\xi}\phi(x,\xi^{\nu}_j)-y|^{-2N}\int_{\mathbb{R}^{n}}\left|\nabla^{2N}_{\xi}(e^{ih_j^{\nu}}a_j\psi^{\nu}_{j})\right|d\xi |b(y)|dy\nonumber\\
 \lesssim &|\nabla_{\xi}\phi(x,\xi^{\nu}_j)|^{-2N}\int_{|y|<r}\int_{A^{\nu}_{j}}2^{j(m_p-2N\rho_0)}d\xi |b(y)|dy\nonumber\\
 \lesssim &2^{j(m_p-2N\rho_0)}r^{n(1-\frac 1p)}|A^{\nu}_{j}||\nabla_{\xi}\phi(x,\xi^{\nu}_j)|^{-2N}.\label{fiohp4.2}
 \end{align}
 It is easy to check that $\rho-(2Np-n+1)\rho_0-n(1-p)\leq -n(1-p)<0$. So, by (\ref{fiohp4.1}), (\ref{fiohp4.2}) and (\ref{fiohp2.1}), we can show that
 \begin{align}
 &\|T_{\phi,a} b\|^p_p\nonumber\\
 \leq &\|T_{\phi,a} b\|^p_{L^p(\widetilde{B}_{r})}+\|T_{\phi,a_0} b\|^p_{L^p(\widetilde{B}^c_{r})}+\|\sum^{\infty}_{j=1}\sum^{J}_{\nu=1}T^{\nu}_{j} b\|^p_{L^p(\widetilde{B}^c_{r})}\nonumber\\
 \lesssim &1+\|T_{\phi,a_0} b\|^p_p+\sum^{\infty}_{j=1}\sum^{J}_{\nu=1}\|T^{\nu}_{j} b\|^p_{L^p(\widetilde{B}^c_{r})}\nonumber\\
 \lesssim &1+\|T_{\phi,a_0} b\|^p_p+\sum^{\infty}_{j=1}\sum^{J}_{\nu=1}2^{j(m_p-2N\rho_0)p}r^{n(p-1)}|A^{\nu}_{j}|^p\int_{|\nabla_{\xi}\phi(x,\xi^{\nu}_j)|>3r}|\nabla_{\xi}\phi(x,\xi^{\nu}_j)|^{-2Np}dx\nonumber\\
 \lesssim &1+\|T_{\phi,a_0} b\|^p_p+\sum^{\infty}_{j=1}2^{j(n-1)\rho_0}2^{j(\rho-n+(n-1-2N)\rho_0p)}2^{j(n-(n-1)\rho_0)p}r^{n-2Np}\nonumber\\
 =&1+\|T_{\phi,a_0} b\|^p_p+\sum^{\infty}_{j=1}2^{j(\rho-(2Np-n+1)\rho_0-n(1-p))}r^{n-2Np}\lesssim 1+\|T_{\phi,a_0} b\|^p_p.\label{fiohp4.3}
 \end{align}
 Therefore, by (\ref{fiohp4.3}) and Theorem C, we finish the proof of Theorem 1.3 when $p<1$ and $r\geq1$. It is worth noting that $p<1$ is necessary in the proof of (\ref{fiohp4.3}).

 \subsection{Case $r<1$}

 Let $j_0$ be the unique positive integer such that $1<2^{j_0}r\leq 2$ and $P_r$ be the "except set" defined in section 2. We divide $\|\sum\limits^{\infty}_{j=1}T_{\phi,a_j} b\|^p_{p}$ into three parts,
 \begin{align}
 \|\sum^{\infty}_{j=1}T_{\phi,a_j}b\|^p_{p}&\leq\sum^{j_0}_{j=1}\sum^{J}_{\nu=1}\|T^{\nu}_{j}b\|^p_{p}
 +\|\sum^{\infty}_{j=j_0+1}T_{\phi,a_j}b\|^p_{L^p(P_r)}+\sum^{\infty}_{j=j_0+1}\sum^{J}_{\nu=1}\|T^{\nu}_{j}b\|^p_{L^p(P_r^c)}\nonumber\\
 &:=I+II+III.\label{fiohp4.4}
 \end{align}

 Take integers $N>\frac np$. Let $N_p$ be the least integer which is greater than $n(\frac 1p-1)$ and $L$ be the operator defined in section 3.2.2.

 When $r|\xi|\leq 1$ and $0\leq |\alpha|<N_p$, by Taylor's formula and the mean value zero of $b$, we have
 \begin{align*}
 |\partial^{\alpha}_{\xi}\widehat{b}(\xi)|&=|\int_{|y|<r}e^{iy\cdot \xi}y^{\alpha}b(y)dy|\\
 &\lesssim |\int_{|y|<r}\sum_{|\beta|<N_p-|\alpha|}\frac{(iy\cdot \xi)^{\beta}}{\beta!}y^{\alpha}b(y)dy|+\int_{|y|<r}|r\xi|^{N_p-|\alpha|}|y|^{|\alpha|}|b(y)|dy\\
 &\lesssim |\xi|^{N_p-|\alpha|}r^{N_p-n(\frac 1p-1)}.
 \end{align*}

 On the other hand, when $r|\xi|\leq 1$ and $|\alpha|\geq N_p$, we also have
 \begin{align*}
 |\partial^{\alpha}_{\xi}\widehat{b}(\xi)|=|\int_{|y|<r}e^{iy\cdot \xi}y^{\alpha}b(y)dy|\lesssim r^{|\alpha|-n(\frac 1p-1)}\lesssim |\xi|^{N_p-|\alpha|}r^{N_p-n(\frac 1p-1)}.
 \end{align*}
 Therefore, when $r|\xi|\leq 1$, for any multi-index $\alpha$, we get that
 \begin{align}
 |\partial^{\alpha}_{\xi}\widehat{b}(\xi)|\lesssim |\xi|^{N_p-|\alpha|}r^{N_p-n(\frac 1p-1)}.\label{fiohp4.5}
 \end{align}

 Then, by using (\ref{fiohp4.5}) and Lemma 2.3, we can show that
 \begin{align*}
 |T^{\nu}_{j}b(x)|=&|\int_{\mathbb{R}^{n}}e^{i\nabla_{\xi}\phi(x,\xi^{\nu}_j)\cdot \xi}e^{ih_j^{\nu}(x,\xi)}a_j(x,\xi)\psi^{\nu}_{j}(\xi)\widehat{b}(\xi)d\xi|\\
 =&(1+2^{2j\rho}|\partial_{\xi^{\nu}_{j}}\phi(x,\xi^{\nu}_j)|^2+2^{2j\rho_0}|\nabla_{\xi}\phi(x,\xi^{\nu}_j)|^2)^{-N}\\
 &|\int_{\mathbb{R}^{n}}L^N(e^{i\nabla_{\xi}\phi(x,\xi^{\nu}_j)\cdot \xi})e^{ih_j^{\nu}(x,\xi)}a_j(x,\xi)\psi^{\nu}_{j}(\xi)\widehat{b}(\xi)d\xi|\\
 \lesssim &(1+2^{j\rho}|\partial_{\xi^{\nu}_{j}}\phi(x,\xi^{\nu}_j)|+2^{j\rho_0}|\nabla_{\xi}\phi(x,\xi^{\nu}_j)|)^{-2N}\int_{\mathbb{R}^{n}}|L^N(e^{ih_j^{\nu}}a_j\psi^{\nu}_{j}\widehat{b})|d\xi\\
 \lesssim&(1+2^{j\rho}|\partial_{\xi^{\nu}_{j}}\phi(x,\xi^{\nu}_j)|+2^{j\rho_0}|\nabla_{\xi}\phi(x,\xi^{\nu}_j)|)^{-2N}\\
 &\sum_{k_1+k_2+k_3\leq 2N}\int_{A^{\nu}_{j}}2^{jk_1\rho}2^{jk_2\rho_0}|\partial^{k_1}_{\xi^{\nu}_{j}}\nabla^{k_2}_{\xi}(e^{ih_j^{\nu}}a_j\psi^{\nu}_{j})|2^{jk_3\rho}|\nabla^{k_3}_{\xi}\widehat{b}(\xi)|d\xi\\
 \lesssim &2^{j(m_p+N_p)}|A^{\nu}_{j}|r^{N_p-n(\frac 1p-1)}(1+2^{j\rho}|\partial_{\xi^{\nu}_{j}}\phi(x,\xi^{\nu}_j)|+2^{j\rho_0}|\nabla_{\xi}\phi(x,\xi^{\nu}_j)|)^{-2N}.
 \end{align*}

 As $N>\frac {2n}{p}$, $N_p>n(\frac 1p-1)$ and $2^{j_0}r\leq 2$, we obtain that
 \begin{align}
 I=&\sum^{j_0}_{j=1}\sum^{J}_{\nu=1}\|T^{\nu}_{j}b\|^p_{p}\nonumber\\
 \lesssim & \sum^{j_0}_{j=1}\sum^{J}_{\nu=1}\int_{\mathbb{R}^{n}}[2^{j(m_p+N_p)}|A^{\nu}_{j}|r^{N_p-n(\frac 1p-1)}
 (1+2^{j\rho}|\partial_{\xi^{\nu}_{j}}\phi(x,\xi^{\nu}_j)|+2^{j\rho_0}|\nabla_{\xi}\phi(x,\xi^{\nu}_j)|)^{-2N}]^pdx\nonumber\\
 \lesssim &\sum^{j_0}_{j=1}\sum^{J}_{\nu=1}2^{j(m_p+N_p)p}|A^{\nu}_{j}|^pr^{N_pp-n(1-p)}2^{-j(\rho+(n-1)\rho_0)}\nonumber\\
 \lesssim &\sum^{j_0}_{j=1}2^{j(n-1)\rho_0}2^{j(\rho-n+(n-1)\rho_0p+N_pp)}2^{j(n-(n-1)\rho_0)p}2^{-j(\rho+(n-1)\rho_0)}r^{N_pp-n(1-p)}\nonumber\\
 = &\sum^{j_0}_{j=1} (2^jr)^{[N_p-n(\frac 1p-1)]p}\lesssim 1.\label{fiohp4.6}
 \end{align}

 For the second term II, take $s,t$ such that $s=(n-\rho)(\frac 1p-1)$ and $\frac 1t=1+\frac s n$. It is easy to see that $p\leq t\leq 1$.
 From the definition of the $L^2$ atom for the Hardy space, one can verify that $r^{n(\frac 1p-\frac 1t)}b$ is a $L^2$-atom for $H^t(\mathbb{R}^n)$. Hence, we obtain
 $$\|b\|_{H^t}\leq r^{n(\frac 1t-\frac 1p)}=r^{-\rho(\frac 1p-1)}.$$
 By Lemma 2.6 we get that
 \begin{align}
 \|I_{s}b\|_{H^1}\lesssim \|b\|_{H^t}\lesssim r^{-\rho(\frac 1p-1)}.\label{fiohp4.7}
 \end{align}
 Set $\widetilde{a}_{j_0}(x,\xi)=a(x,\xi)(1-\Psi_0(2^{-j_0}\xi))|\xi|^{s}$. It is easy to check that $\widetilde{a}_{j_0}\in L^{\infty}S^{\rho+(n-1)\rho_0-n}_{\rho}$ and $T_{\phi,\widetilde{a}_{j_0}}$
 is bounded from $H^1$ to $L^1$ (proved in Section 3). Therefore, by (\ref{fiohp2.4}) and (\ref{fiohp4.7}), we get that
 \begin{align}
 II&=\|\sum^{\infty}_{j=j_0+1}T_{\phi,a_j}b\|_{L^p(P_r)}=\|T_{\phi,\widetilde{a}_{j_0}}I_{s}b\|_{L^p(P_r)}\nonumber\\
 \leq &|P_r|^{\frac 1p-1}\|T_{\phi,\widetilde{a}_{j_0}}I_{s}b\|_1\lesssim r^{\frac 1p-1}\|I_{s}b\|_{H^1}\nonumber\\
 \lesssim &r^{(\frac 1p-1)(1-\rho)}\lesssim 1.\label{fiohp4.8}
 \end{align}

 From (\ref{fiohp2.5}) and Lemma 2.3, when $x\in P_r^c$ we can show that
 \begin{align*}
 &|T^{\nu}_{j}b(x)|=|\int_{|y|<r}\int_{\mathbb{R}^{n}}e^{i(\nabla_{\xi}\phi(x,\xi^{\nu}_j)-y)\cdot \xi}e^{ih_j^{\nu}(x,\xi)}a_j(x,\xi)\psi^{\nu}_{j}(\xi)d\xi b(y)dy|\\
 =&|\int_{|y|<r}(1+2^{2j\rho}|\xi^{\nu}_j\cdot(\nabla_{\xi}\phi(x,\xi^{\nu}_j)-y)|^2+2^{2j\rho_0}|\nabla_{\xi}\phi(x,\xi^{\nu}_j)-y|^2)^{-N}\\
 &\int_{\mathbb{R}^{n}}L^N(e^{i(\nabla_{\xi}\phi(x,\xi^{\nu}_j)-y)\cdot \xi})e^{ih_j^{\nu}(x,\xi)}a_j(x,\xi)\psi^{\nu}_{j}(\xi)d\xi b(y)dy|\\
 \lesssim &\int_{|y|<r}(1+2^{j\rho}|\xi^{\nu}_j\cdot(\nabla_{\xi}\phi(x,\xi^{\nu}_j)-y)|+2^{j\rho_0}|\nabla_{\xi}\phi(x,\xi^{\nu}_j)-y|)^{-2N}\\
 &\sum_{k_1+k_2\leq 2N}\int_{A^{\nu}_{j}}2^{jk_1\rho}2^{jk_2\rho_0}|\partial^{k_1}_{\xi^{\nu}_{j}}\nabla^{k_2}_{\xi}(e^{ih_j^{\nu}}a_j\psi^{\nu}_{j})|d\xi |b(y)|dy\\
 \lesssim &2^{jm_p}|A^{\nu}_{j}|r^{n(1-\frac 1p)}(1+2^{j\rho}|\partial_{\xi^{\nu}_{j}}\phi(x,\xi^{\nu}_j)|+2^{j\rho_0}|\nabla_{\xi}\phi(x,\xi^{\nu}_j)|)^{-2\lambda_{\rho}N}.
 \end{align*}

  As $p<1$ and $2\lambda_{\rho}Np\geq Np>n$, by (\ref{fiohp2.1}) we get that
 \begin{align}
 III=&\sum^{\infty}_{j=j_0+1}\sum^{J}_{\nu=1}\|T^{\nu}_{j}b\|^p_{L^p(P^c_r)}\nonumber\\
 \lesssim &\sum^{\infty}_{j=j_0+1}\sum^{J}_{\nu=1}2^{jm_pp}|A^{\nu}_{j}|^pr^{n(p-1)}
 \int_{\mathbb{R}^{n}}(1+2^{j\rho}|\partial_{\xi^{\nu}_{j}}\phi(x,\xi^{\nu}_j)|+2^{j\rho_0}|\nabla_{\xi}\phi(x,\xi^{\nu}_j)|)^{-2\lambda_{\rho}Np}dx\nonumber\\
 \lesssim &\sum^{\infty}_{j=j_0+1}\sum^{J}_{\nu=1}2^{j(\rho-n+(n-1)\rho_0p)}2^{j(n-(n-1)\rho_0)p}r^{n(p-1)}\int_{\mathbb{R}^{n}}(1+2^{j\rho}|\xi^{\nu}_j\cdot z|+2^{j\rho_0}|z|)^{-2\lambda_{\rho}Np}dz\nonumber\\
 \lesssim &\sum^{\infty}_{j=j_0+1}(2^{j}r)^{n(p-1)}\lesssim 1.\label{fiohp4.9}
 \end{align}
 In the proof of (\ref{fiohp4.9}), $p<1$ is necessary. Now, by (\ref{fiohp4.4}), (\ref{fiohp4.6}), (\ref{fiohp4.8}) and (\ref{fiohp4.9}) we get
 \begin{align}
 \|\sum^{\infty}_{j=1}T_{\phi,a}b\|^p_{p}\leq I+II+III\lesssim 1.\label{fiohp4.10}
 \end{align}
 Therefore, by (\ref{fiohp4.10}) and Theorem C, we finish the proof of Theorem 1.3 when $p<1$ and $r<1$.

 \section{Appendix A: proof of Lemma 2.1}

 \hspace{4mm}  Thank to \cite[Theorem 1.2]{GZ22} and the theory of atom decomposition for $h^1$, it is left for us to show $\|T_{a} b\|_1\lesssim 1$ if $b$ is supported in $B_r$ with $r\geq 1$   and $\|b\|_2\leq r^{-n/2}$.

 We first use the $L^2$ boundedness of $T_a$ (ensured by \cite[Proposition 2.3]{KS07}) to get
 \begin{align*}
 \|T_a b\|_{L^1(B_{2r})}\lesssim r^{n/2}\|T_ab\|_2\lesssim r^{n/2}\|b\|_2\lesssim 1.
 \end{align*}
 By integration by parts, when $x\neq 0$ we obtain that
 \begin{align*}
 |T_a b(x)|\lesssim &|x|^{-n}\sum_{|\alpha|+|\beta|=n}\left|\int_{\mathbb{R}^n} e^{ix\cdot\xi}\partial^{\alpha}_{\xi}a(x,\xi)\partial^{\beta}_{\xi}\widehat{b}(\xi)d\xi\right|\\
 \lesssim &|x|^{-n}\sum_{|\alpha|+|\beta|=n}\left|\int_{\mathbb{R}^n} e^{ix\cdot\xi}\partial^{\alpha}_{\xi}a(x,\xi)\widehat{y^{\beta} b}(\xi)d\xi\right|\\
 =&|x|^{-n}\sum_{|\alpha|+|\beta|=n}|T_{a_{\alpha}}(y^{\beta} b)(x)|,
 \end{align*}
 where $T_{a_\alpha}$ is the pseudo-differential operator with the amplitude $a_{\alpha}=\partial^{\alpha}_{\xi}a$. It is obvious that $a_{\alpha}\in L^\infty S^{-n}_0$. Thus $T_{a_{\alpha}}$ is bounded on $L^2$ and
 \begin{align*}
 \int_{|x|\geq 2r} |T_a b(x)|dx
 \lesssim &\sum_{|\alpha|+|\beta|=n}\int_{|x|\geq 2r} |x|^{-n}|T_{a_{\alpha}}(y^{\beta} b)(x)|dx\\
 \lesssim &\sum_{|\alpha|+|\beta|=n}r^{-n/2}\|y^{\beta}b(y)\|_{L^2(dy)} \\
 \lesssim &\sum_{|\beta|\leq n}r^{-n+|\beta|}\lesssim  1.
 \end{align*}
 Therefore, we get that $\|T_a b\|_1\lesssim 1$ if $r\geq1$ and this finishes the proof of Lemma 2.1.

 \section{Appendix B: proof of Lemma 2.2}

 \hspace{4mm}  When $j>0$ and $\xi\in A^{\nu}_{j}$, it is easy to see that
 $$|\xi|\approx \xi^{\nu}_{j}\cdot \xi\approx2^{j}\quad \textrm{and}\quad |\xi-(\xi^{\nu}_{j}\cdot \xi)\xi^{\nu}_{j}|\lesssim 2^{j(1-\rho_0)}.$$
 For any $\lambda>0$, since $\phi$ is positively homogeneous of degree 1 in $\xi$, one can verify that $\phi(x,\xi)=\xi\cdot \nabla_{\xi}\phi(x,\xi)$ and then
 \begin{align*}
 &h^{\nu}_{j}(x,\lambda\xi^{\nu}_{j})=\phi(x,\lambda\xi^{\nu}_{j})-\lambda\xi^{\nu}_{j}\cdot \nabla_{\xi}\phi(x,\xi^{\nu}_{j})=\lambda\phi(x,\xi^{\nu}_{j})-\lambda\phi(x,\xi^{\nu}_{j})=0,\\
 &(\nabla_{\xi}h^{\nu}_{j})(x,\lambda\xi^{\nu}_{j})=(\nabla_{\xi}\phi)(x,\lambda\xi^{\nu}_{j})-\nabla_{\xi}\phi(x,\xi^{\nu}_{j})=\nabla_{\xi}\phi(x,\xi^{\nu}_{j})-\nabla_{\xi}\phi(x,\xi^{\nu}_{j})=0.
 \end{align*}
 So for any $N\geq 0$ and $\lambda>0$, one can get that
 \begin{align*}
 \partial^{N}_{\xi^{\nu}_{j}}h^{\nu}_{j}(x,\lambda\xi^{\nu}_{j})=0,\partial^{N}_{\xi^{\nu}_{j}}\nabla_{\xi}h^{\nu}_{j}(x,\lambda\xi^{\nu}_{j})=0.
 \end{align*}
 When $M\geq2,N\geq 0$, by the positively homogeneity of $\phi$ and the fact $\rho_0\leq \frac 12$, we get
 \begin{align}
 |\partial^{N}_{\xi^{\nu}_{j}}\nabla^{M}_{\xi}h^{\nu}_{j}(x,\xi)|=|\partial^{N}_{\xi^{\nu}_{j}}\nabla^{M}_{\xi}\phi(x,\xi)|\lesssim 2^{j(1-N-M)}
 \lesssim 2^{-j(N+M\rho_0)}. \label{fiohp6.1}
 \end{align}
 When $M=1$, $N\geq 0$, by using $(\partial^{N}_{\xi^{\nu}_{j}}\nabla_{\xi}h^{\nu}_{j})(x,(\xi^{\nu}_{j}\cdot \xi)\xi^{\nu}_{j})=0$, the mean value theorem, the bound \eqref{fiohp6.1}
 and $|\xi-(\xi^{\nu}_{j}\cdot \xi)\xi^{\nu}_{j}|\lesssim 2^{j(1-\rho_0)}$ we have
 \begin{align*}
 |(\partial^{N}_{\xi^{\nu}_{j}}\nabla_{\xi}h^{\nu}_{j})(x,\xi)|&=
 |(\partial^{N}_{\xi^{\nu}_{j}}\nabla_{\xi}h^{\nu}_{j})(x,\xi)-(\partial^{N}_{\xi^{\nu}_{j}}\nabla_{\xi}h^{\nu}_{j})(x,(\xi^{\nu}_{j}\cdot \xi)\xi^{\nu}_{j})|\\
 &\lesssim 2^{-j(N+1)}2^{j(1-\rho_0)}=2^{-j(N+\rho_0)}.
 \end{align*}
 When $M=0$, $N\geq 1$, it is easy to check that $1-N-2\rho_0\leq -N\rho$. Therefore, applying similar arguments, we can conclude that
 \begin{align*}
 |\partial^{N}_{\xi^{\nu}_{j}}h^{\nu}_{j}(x,\xi)|
 &=|\partial^{N}_{\xi^{\nu}_{j}}h^{\nu}_{j}(x,\xi)-\partial^{N}_{\xi^{\nu}_{j}}h^{\nu}_{j}(x,(\xi^{\nu}_{j}\cdot \xi)\xi^{\nu}_{j})|\\
 &\lesssim 2^{-j(N+\rho_0)}2^{j(1-\rho_0)}\leq 2^{-jN\rho}.
 \end{align*}
 Therefore we get desired estimates for $h^{\nu}_{j}$.

 On the other hand, when $\xi\in A^{\nu}_{j}$ with $j>0$, we claim that
 \begin{align}
 \partial_{\xi^{\nu}_{j}}=\partial_{r}+O(2^{-j\rho_0})\cdot\nabla_{\xi},   \label{fiohp6.2}
 \end{align}
 where $\partial_r$ is the radial derivative. Indeed, by using the polar coordinates $\xi=r\theta$ with $r=|\xi|$, we get
 \begin{align*}
 &0\leq |\xi|-\xi^{\nu}_{j}\cdot \xi=\frac{|\xi-(\xi^{\nu}_{j}\cdot \xi)\xi^{\nu}_{j}|^2}{|\xi|+\xi^{\nu}_{j}\cdot \xi}\lesssim 2^{j(1-2\rho_0)},\\
 &\partial_{r}=O(1)\cdot\nabla_{\xi}, \qquad  \partial_{\theta}=O(2^{j})\cdot\nabla_{\xi},
 \end{align*}
 and
 \begin{equation*}
 |\frac{\partial \theta}{\partial \xi^{\nu}_{j}}|=\frac{\left||\xi|\frac{\partial \xi}{\partial \xi^{\nu}_{j}}-\frac{\partial |\xi|}{\partial \xi^{\nu}_{j}}\xi\right|}{|\xi|^2}
 \leq \frac{|\xi-(\xi^{\nu}_{j}\cdot \xi)\xi^{\nu}_{j}|^2+|\xi^{\nu}_{j}\cdot \xi||\xi-(\xi^{\nu}_{j}\cdot \xi)\xi^{\nu}_{j}|}{|\xi|^3}\lesssim 2^{-j(1+\rho_0)}.
 \end{equation*}
 Thus when $\xi\in A^{\nu}_{j}$ with $j>0$, we can write
 \begin{align*}
 \partial_{\xi^{\nu}_{j}}&=\frac{\partial r}{\partial \xi^{\nu}_{j}}\partial_{r}+\frac{\partial \theta}{\partial \xi^{\nu}_{j}}\partial_{\theta}
 =\frac{\xi^{\nu}_{j}\cdot \xi}{|\xi|}\partial_{r}+O\left(2^{-j(1+\rho_0)}\right) O(2^j)\cdot\nabla_{\xi}\nonumber\\
 &=\partial_r+\frac{\xi^{\nu}_{j}\cdot \xi-|\xi|}{|\xi|}\partial_{r}+O(2^{-j\rho_0})\cdot\nabla_{\xi}=\partial_{r}+O(2^{-j\rho_0})\cdot\nabla_{\xi}.
 \end{align*}
 So the claim is true.

 Since $\psi^{\nu}_{j}$ is homogeneous of degree 0, we have $\partial^{k}_{r}\psi^{\nu}_{j}=0$ for any $k>0$.
 Now \eqref{fiohp2.2} and \eqref{fiohp6.2} yield that
 \begin{align*}
 \left|\partial^{N}_{\xi^{\nu}_{j}}\nabla^{M}_{\xi}\psi^{\nu}_{j}(\xi)\right|
 &=\left|\left(\partial_{r}+O\left(2^{-j\rho_0}\right)\cdot\nabla_{\xi}\right)^{N}\nabla^{M}_{\xi}\psi^{\nu}_{j}(\xi)\right|\\
 &\lesssim 2^{-jN\rho_0}\left|\nabla^{N+M}_{\xi}\psi^{\nu}_{j}(\xi)\right|\\
 &\lesssim 2^{-jN\rho_0}2^{-j(N+M)}2^{j(N+M)\rho_0}\\
 &=2^{-j(N+(1-\rho_0)M)}\leq 2^{-j(N+\rho_0M)}
 \end{align*}
 as desired.

 \section{Appendix C: proof of Lemma 2.5}

 \hspace{4mm}  When $j=0$, one can see \cite[Theorem 1.18]{FS14}.

 When $j>0$, set
 $$S_{\phi,a_j}g(x)=\int_{\mathbb{R}^{n}}e^{i\phi(x,\xi)}a_j(x,\xi)g(\xi)d\xi.$$

 It is easy to see that $S_{\phi,a_j}S^*_{\phi,a_j}$ can be given by
 $$S_{\phi,a_j}S^*_{\phi,a_j}f(x)=\int_{\mathbb{R}^n} K_j(x,y)f(y)dy,$$
 where
 $$K_j(x,y)=\int_{\mathbb{R}^n}e^{i(\phi(x,\xi)-\phi(y,\xi))}a_j(x,\xi)\overline{a_j(y,\xi)}d\xi.$$

 Set $G(x,y,\xi)=\phi(x,\xi)-\phi(y,\xi)$.  Define the operator $L$ as
 $$L=\frac{\nabla_{\xi} G\cdot \nabla_{\xi}}{|\nabla_{\xi} G|^2}$$
 and let $L^*$ be the dual of $L$.

 As $\phi\in \Phi^2$, for any $k\geq 1$,  we have
 \begin{align}\label{fiohp7.1}
 |\nabla^k_{\xi} G(x,y,\xi)|\leq |x-y|\sup_{(z,\xi)\in\mathbb{R}^{n}\times(\mathbb{R}^{n}\setminus\{0\})}|\nabla_z\nabla^k_{\xi} \phi(z,\xi)|\lesssim |x-y||\xi|^{1-k}.
 \end{align}
 On the other hand, as $\phi\in \Phi^{2}$ satisfies the SND condition (\ref{fiohp1.3}), by using \cite[Proposition 1.11]{FS14}, one can get that
 \begin{align}\label{fiohp7.2}
 |\nabla_{\xi} G(x,y,\xi)|=|\nabla_{\xi}\phi(x,\xi)-\nabla_{\xi}\phi(y,\xi)|\geq c|x-y|,
 \end{align}
 where $c$ depends only on $\lambda$.

 One can easily see that
 $$L(e^{i(\phi(x,\xi)-\phi(y,\xi))})=\frac{\nabla_{\xi} G\cdot \nabla_{\xi}}{|\nabla_{\xi} G|^2}e^{iG}=ie^{iG}=ie^{i(\phi(x,\xi)-\phi(y,\xi))}.$$
 So, for any $M\in \mathbb{N}$, applying Lemma 3.2 in \cite{RZ23}, from (\ref{fiohp7.1}), (\ref{fiohp7.2}) and the fact $a\in L^{\infty}S^{m}_{\rho}$, we have
 \begin{align}
 &\left|K_j(x,y)\right|\nonumber\\
 =&\left|\int_{\mathbb{R}^n}e^{i(\phi(x,\xi)-\phi(y,\xi))}a_j(x,\xi)\overline{a_j(y,\xi)}d\xi\right|\nonumber\\
 =&\left|\int_{\mathbb{R}^n}(-i)^ML^M(e^{i(\phi(x,\xi)-\phi(y,\xi))})a_j(x,\xi)\overline{a_j(y,\xi)}d\xi\right|\nonumber\\
 =&\left|\int_{\mathbb{R}^n}e^{i(\phi(x,\xi)-\phi(y,\xi))}(L^*)^M(a_j(x,\xi)\overline{a_j(y,\xi)})d\xi\right|\nonumber\\
 \lesssim &\int_{\mathbb{R}^n}|\nabla_{\xi} G|^{-4M}\sum_{k_0+k_1+\cdots+k_{3M}=M,k_s\geq0}\left|\nabla_{\xi}^{k_0}(a_j(x,\xi)\overline{a_j(y,\xi)})\right|\prod^{3M}_{s=1}\left|\nabla_{\xi}^{k_s+1}G\right|d\xi\nonumber\\
 \lesssim &\int_{B_{2^{j+1}}\setminus B_{2^{j-1}}}|x-y|^{-4M}\sum_{k_0+k_1+\cdots+k_{3M}=M,k_s\geq 0}|\xi|^{2m-k_0\rho}\prod^{3M}_{s=1}\left(|x-y||\xi|^{-k_s}\right)d\xi\nonumber\\
 \lesssim &2^{jn}|x-y|^{-4M}\sum^M_{k_0=0}2^{j(2m-k_0\rho)}|x-y|^{3M}2^{j(k_0-M)}\nonumber\\
 \lesssim &2^{j(n+2m)}(2^{j\rho}|x-y|)^{-M}.\label{fiohp7.3}
 \end{align}

 For any $x$, letting $M=0$ or $n+1$ in (\ref{fiohp7.3}), one can obtain that
 \begin{align}
 &\int_{\mathbb{R}^n}|K_j(x,y)|dy\nonumber\\
 =&\int_{|y-x|<2^{-j\rho}}|K_j(x,y)|dy+\int_{|y-x|\geq 2^{-j\rho}}|K_j(x,y)|dy\nonumber\\
 \lesssim &\int_{|y-x|<2^{-j\rho}}2^{j(n+2m)}dy+\int_{|y-x|\geq 2^{-j\rho}}2^{j(n+2m)}(2^{j\rho}|x-y|)^{-n-1}dy\nonumber\\
 \lesssim & 2^{j(2m-n(\rho-1))}.\label{fiohp7.4}
 \end{align}
 Similarly, for any $y$, we get
 $$\int_{\mathbb{R}^n}|K_j(x,y)|dx\lesssim 2^{j(2m-n(\rho-1))}.$$
 Therefore,
 $$\|S_{\phi,a_j}S^*_{\phi,a_j}\|_{L^1-L^1}\lesssim 2^{j(2m-n(\rho-1))},\|S_{\phi,a_j}S^*_{\phi,a_j}\|_{L^{\infty}-L^{\infty}}\lesssim 2^{j(2m-n(\rho-1))},$$
 which implies that
 $$\|S_{\phi,a_j}S^*_{\phi,a_j}\|_{L^2-L^2}\lesssim 2^{j(2m-n(\rho-1))}.$$
 By a standard dual argument, we have
 $$\|S_{\phi,a_j}\|_{L^2-L^2}=\|S_{\phi,a_j}S^*_{\phi,a_j}\|^{\frac 12}_{L^2-L^2}\lesssim 2^{j(m-\frac{n}{2}(\rho-1))}.$$
 It follows from the Plancherel theorem that
 \begin{align*}
 \|T_{\phi,a_j}f\|_2=\|S_{\phi,a_j}\widehat{f}\|_2\lesssim 2^{j(m-\frac{n}{2}(\rho-1))}\|\widehat{f}\|_2=2^{j(m-\frac{n}{2}(\rho-1))}\|f\|_2.
 \end{align*}
 This finishes the proof.

 \section*{}

 \noindent Xiaofeng Ye\\
 Department of Mathematics and Information Computing,\\ East China Jiaotong University, Nanchang, 310013, P. R. China\\
  Email: xiaofye@163.com\\

 \noindent Chunjie Zhang\\
 Zhejiang University Press, Hangzhou, 310028, China\\
 Email: cjzhang@zju.edu.cn\\

 \noindent Xiangrong Zhu\\
 School of Mathematical Sciences, Zhejiang Normal University,\\  Jinhua 321004, P.R. China\\
 Email: zxr@zjnu.cn


\begin{thebibliography}{99}
 \bibitem{AH90} \'{A}lvarez J. and Hounie J.,
 Estimates for the kernel and continuity properties of pseudo-differential operators,
 \textbf{Ark. Mat.} 28 (1990),  1-22.

 \bibitem{Af78} Asada K. and Fujiwara D.,
 On some oscillatory integral transformations in $L^{2}(\mathbb{R}^n)$,
 \textbf{Japan. J. Math. (N.S.).} 4 (1978), 299-361.

 \bibitem{B97} Boulkhemair A.,
 Estimations $L^{2}$ pr\'{e}cis\'{e}es pour des int\'{e}grales oscillantes, (French) [$L^{2}$-estimates for oscillating integrals],
 \textbf{Comm. Partial Differential Equations}  22 (1997), 165-184.

 \bibitem{CV71} Calder\'{o}n A. P. and Vaillancourt R.,  On the boundedness of pseudo-differential operators,
 \textbf{ J. Math. Soc. Japan} 23 (1971), 374-378.

 \bibitem{CV72} Calder\'{o}n A. P. and Vaillancourt R., A class of bounded pseudo-differential operators, \textbf{Proc. Nat. Acad. Sci. U.S.A.} 69 (1972), 1185-1187.

 \bibitem{CIS21} Castro A., Israelsson A. and Staubach  W.,
 Regularity of Fourier integral operators with amplitudes in general H\"{o}rmander classes,
 \textbf{Anal. Math. Phys.} 11 (2021), no.3, Paper No. 121, 54 pp.

 \bibitem{CNR09} Cordero E., Nicola F. and Rodino L.,
 Boundedness of Fourier integral operators on $\mathcal{F}L^p$ spaces,
 \textbf{Trans. Amer. Math. Soc.} 361 (2009), 6049-6071.

 \bibitem{CNR10} Cordero E.,  Nicola F. and Rodino L.,
 On the global boundedness of Fourier integral operators,
 \textbf{Ann. Global Anal. Geom.} 38 (2010), 373-398.

 \bibitem{CR10} Coriasco S. and Ruzhansky M.,
 On the boundedness of Fourier integral operators on $L^{p}(\mathbb{R}^n)$,
 \textbf{C. R. Math. Acad. Sci. Paris.} 348 (2010), 847-851.

 \bibitem{CR14} Coriasco S. and Ruzhansky M.,
 Global $L^p$-continuity of Fourier integral operators,
 \textbf{Trans. Amer. Math. Soc.} 366 (2014), 2575-2596.

 \bibitem{DGZ23} Dai J., Guo J. and Zhu X.,
 $L^p$ boundedness of Fourier integral operators with rough symbols,
 \textbf{J. Math. Anal. Appl.} 517 (2023), 126654 (14pp).

 \bibitem{FS14} Dos Santos Ferreira D. and Staubach W.,
 Global and local regularity of Fourier integral operators on weighted and unweighted spaces,
 \textbf{Mem. Amer. Math. Soc. } 229 (2014), xiv+65pp.

 \bibitem{E70} \`{E}skin G.I.,
 Degenerate elliptic pseudodifferential equations of principal type (Russian),
 \textbf{Mat. Sb. (N.S.).} 82 (1970), no. 124, 585-628.

 \bibitem{F78} Fujiwara D.,
 A global version of Eskin's theorem,
 \textbf{J. Fac. Sci. Univ. Tokyo Sect. IA Math.} 24 (1977), 327-339.

 \bibitem{G79} Goldberg D., A local version of real Hardy spaces, \textbf{Duke Math. J.} 46 (1979), no. 1, 27-42.

 \bibitem{GZ22} Guo  J. and Zhu X.,
 Some notes on endpoint estimates for pseudo-differential operators,
 \textbf{Mediterr. J. Math.} 19 (2022), no. 6, Paper No. 260, 14 pp.

 \bibitem{HPR20} Hassell A.,  Portal P. and Rozendaal J.,
 Off-singularity bounds and Hardy spaces for Fourier integral operators,
 \textbf{Trans. Amer. Math. Soc.} 373 (2020),  5773-5832.

 \bibitem{H71CPAM} H{\"o}rmander L.,
 On the $L^{2}$ continuity of pseudo-differential operators,
 \textbf{Comm. Pure Appl. Math.} 24 (1971), 529-535.

 \bibitem{H71AM} H{\"o}rmander L.,  Fourier integral operators. I,
\textbf{Acta Math.} 127 (1971), 79-183.

 \bibitem{H86} Hounie J.,
 On the $L^2$-continuity of pseudo-differential operators,
 \textbf{Comm. Partial Differential equations} 11 (1986), 765-778.

 \bibitem{IRS21} Israelsson A., Rodr\'{i}guez-L\'{o}pez S. and Staubach W.,
 Local and global estimates for hyperbolic equations in Besov-Lipschitz and Triebel-Lizorkin spaces,
 \textbf{Anal. PDE.} 14 (2021), no.1, 1-44.

 \bibitem{IMS23} Israelsson A., Mattsson T. and Staubach W.,
 Boundedness of Fourier integral operators on classical function spaces,
 \textbf{J. Funct. Anal.} 285 (2023), no. 5, Paper No. 110018, 64 pp.

\bibitem{K82} Krantz S. G., Fractional integration on Hardy spaces,
\textbf{Studia Math.} 73 (1982), no. 2, 87-94.

 \bibitem{KS07} Kenig C.E., Staubach W.,
 $\Psi$-pseudodifferential operators and estimates for maximal oscillatory integrals,
 \textbf{Studia Math.} 183 (2007), 249-258.

 \bibitem{K76} Kumano-go H.,
 A calculus of Fourier integral operators on $\mathbb{R}^n$ and the fundamental solution for an operator of hyperbolic type,
 \textbf{Comm. Partial Differential Equations} 1 (1976), 1-44.

 \bibitem{R76} Rodino L.,
 On the boundedness of pseudo differential operators in the class $L^{m}_{\rho,1}$,
 \textbf{Proc. Amer. Math. Soc. } 58 (1976), 211-215.

 \bibitem{RZ23} Ruan J. and Zhu X.,
 Fourier integral operators with forbidden symbols on the Besov spaces,
 \textbf{To appear in Forum Math.} https://doi.org/10.1515/forum-2023-0092.

 \bibitem{RS06} Ruzhansky M. and Sugimoto M.,
 Global $L^{2}$-boundedness theorems for a class of Fourier integral operators,
 \textbf{Comm. Partial Differential Equations} 31 (2006), 547-569.

 \bibitem{RS19} Ruzhansky M. and Sugimoto M.,
 A local-to-global boundedness argument and Fourier integral operators,
 \textbf{J. Math. Anal. Appl.} 473 (2019), no. 2, 892-904.

 \bibitem{SSS91} Seeger A., Sogge C. D. and Stein E. M., Regularity properties of Fourier integral operators,
 \textbf{Ann. of Math.}  134 (1991), 231-251.

 \bibitem{SZ23} Shen S. and Zhu X.,
 Fourier integral operators on $L^p(\mathbb{R}^n)$ when $2<p\leq\infty$,
 \textbf{Anal. Math. Phys.} 13 (2023), no. 4, Paper No. 62, 20 pp.

 \bibitem{S93} Stein E. M.,
 Harmonic analysis: real-variable methods, orthogonality, and oscillatory integrals, With the assistance of Timothy S. Murphy,
 \textbf{Princeton Mathematical Series, 43. Monographs in Harmonic Analysis, III.} Princeton University Press, Princeton, NJ, 1993.

 \bibitem{T04} Tao T., The weak-type $(1,1)$ of Fourier integral operators of order $-(n-1)/2$,
 \textbf{J. Aust. Math. Soc.} 76 (2004), 1-21.
\end{thebibliography}
\end{document}